\documentclass[10pt]{article}

\usepackage{amssymb,amsthm,amsmath,tikz, booktabs,algorithm,algpseudocode, multirow}
\usepackage{algorithmicx}
\algtext*{EndIf}
\algtext*{EndFor}
\algtext*{EndWhile}

\usepackage{filecontents}
\usepackage{pgfplots}
\usepackage{pgfplotstable}
\usepackage{authblk}
\usepackage{subfigure}
\usepackage{caption}
\usetikzlibrary{decorations.pathmorphing}
\usetikzlibrary{decorations,arrows,patterns}
\usepgflibrary{decorations.pathreplacing} 

\newtheorem{Theorem}{Theorem}
\newtheorem{Lemma}{Lemma}
\newtheorem{Corollary}{Corollary}
\newtheorem{Definition}{Definition}
\newtheorem{Remark}{Remark}
\newtheorem{Example}{Example}
\newtheorem{Conclusion}{Conclusion}

\newcommand\numberthis{\addtocounter{equation}{1}\tag{\theequation}}

\begin{document}
\title{A flow based pruning scheme for enumerative equitable coloring algorithms}
\author{A.M.C.A. Koster$^{\star}$, R. Scheidweiler$^{\star}$, and M. Tieves\footnote{$\left\{koster,scheidweiler,tieves\right\}$@math2.rwth-aachen.de} } 
\affil{Lehrstuhl II f\"ur Mathematik, RWTH Aachen University, \\ 52056 Aachen, Germany } 

\date{}
\maketitle

\begin{abstract}
An equitable graph coloring is a proper vertex coloring of a graph $G$ where the sizes of the color classes differ by at most one. The equitable chromatic number, denoted by $\chi_{eq}(G),$ is the smallest number $k$ such that $G$ admits such equitable $k$-coloring. 

We focus on enumerative algorithms for the computation of $\chi_{eq}(G)$ and propose a general scheme to derive pruning rules for them: We show how the extendability of a partial coloring into an equitable coloring can be modeled via network flows. Thus, we obtain pruning rules which can be checked via flow algorithms. Computational experiments show that the search tree of enumerative algorithms can be significantly reduced in size by these rules and, in most instances, such naive approach even yields a faster algorithm. Moreover, the stability, i.e., the number of solved instances within a given time limit, is greatly improved. \\
Since the execution of flow algorithms at each node of a search tree is time consuming, we derive arithmetic pruning rules (generalized Hall-conditions) from the network model. Adding these rules to an enumerative algorithm yields an even larger runtime improvement.
\end{abstract}

\section{Introduction}

The Vertex Coloring Problem (VCP) arises in numerous applications. For example, consider the case of assigning a set $V$ of $n$ jobs to $k$ identical machines. We assume that mutual conflicts between the jobs exist, so that, if $v,w\in V$ are in conflict, they cannot be assigned to the same machine. We represent these relations among the jobs by a graph $G=(V,E)$, denoting jobs as vertices with an edge between them if the corresponding tasks are in conflict. Given some $k\in\mathbb{N}$, VCP asks whether the jobs can be partitioned into $k$ sets $C_1,\dots,C_k$ of pairwise nonadjacent vertices each to be processed on the same machines. 

The chromatic number $\chi(G)$ defines the smallest number of machines required to process all the jobs at once. More formally, the graph $G$ admits a $k$-coloring if $V$ can be partitioned into $k$ sets $C_1,\dots,C_k$ of pairwise nonadjacent vertices and the chromatic number $\chi(G)$ is the smallest number $k$ such that $G$ has a $k$-coloring. The sets $C_i$ are called color classes and $\Pi:=(C_1,\dots,C_k,C_{k+1},\dots,C_n)$ with $C_{k+1}=\dots=C_n=\emptyset$ is called a $k$-coloring. 

With respect to fairness, i.e., concerning machine usage, a natural constraint is to distribute the number of jobs equally among the machines. That is, given some $k\in\mathbb{N}$, the task is to partition $V$ into stable sets $C_1,\dots, C_k$ such that their sizes differ by at most one. We refer to this task as the Equitable Graph Coloring Problem (ECP). The equitable chromatic number $\chi_{eq}(G)$ is defined as the smallest integer $k$ such that $G$ possesses a $k$-coloring $\Pi_{eq}$ where $||C_i|-|C_j||\leq 1$ for all color classes $i,j \in \{1,\dots,k\}$.

From an economical point of view, the knowledge of $\chi_{eq}(G)$ is important in a number of applications, e.g., in scheduling problems or, even more applied, in municipal services such as garbage collection, see~\cite{T:1973}. 

\begin{figure}
\centering
\subfigure[{Ordinary graph coloring}]{
\begin{tikzpicture}[rotate=90,scale=0.5,>=latex]
{\tikzstyle{every node}=[draw,shape=circle, inner sep=2pt];   
\path (3.5,0) node[fill=blue!50] (s){\tiny 1};}
{\tikzstyle{every node}=[draw,shape=circle, inner sep=2pt];   
\path (6,-5) node[fill=green] (u1){\tiny 2};
\path (6,-4) node[fill=green] (u2){\tiny 2};
\path (6,-3) node[fill=green] (u3){\tiny 2};
\path (6,-2) node[fill=green] (u4){\tiny 2};
\path (6,-1) node[fill=green] (u5){\tiny 2};
\path (6,0) node[fill=green] (u6){\tiny 2};
\path (6,1) node[fill=green] (u7){\tiny 2};
\path (6,2) node[fill=green] (u8){\tiny 2};
\path (6,3) node[fill=green] (u9){\tiny 2};
\path (6,4) node[fill=green] (u10){\tiny 2};
\path (6,5) node[fill=green] (u11){\tiny 2};
}
\draw (s) to (u1);
\draw (s) to (u2);
\draw (s) to (u3);
\draw (s) to (u4);
\draw (s) to (u5);
\draw (s) to (u6);
\draw (s) to (u7);
\draw (s) to (u8);
\draw (s) to (u9);
\draw (s) to (u10);
\draw (s) to (u11);
\end{tikzpicture}
}\hfill
\centering
\subfigure[{Equitable graph coloring}]{
\begin{tikzpicture}[rotate=90,scale=0.5,>=latex]
{\tikzstyle{every node}=[draw,shape=circle, inner sep=2pt];   
\path (3.5,0) node[fill=blue!50] (s){\tiny 1};}
{\tikzstyle{every node}=[draw,shape=circle, inner sep=2pt];   
\path (6,-5) node[fill=green] (u1){\tiny 2};
\path (6,-4) node[fill=green] (u2){\tiny 2};
\path (6,-3) node[fill=red] (u3){\tiny 3};
\path (6,-2) node[fill=red] (u4){\tiny 3};
\path (6,-1) node[fill=yellow] (u5){\tiny 4};
\path (6,0) node[fill=yellow] (u6){\tiny 4};
\path (6,1) node[fill=orange] (u7){\tiny 5};
\path (6,2) node[fill=orange] (u8){\tiny 5};
\path (6,3) node[fill=brown] (u9){\tiny 6};
\path (6,4) node[fill=brown] (u10){\tiny 6};
\path (6,5) node[fill=purple] (u11){\tiny 7};
}
\draw (s) to node[pos=0.5,below]{} (u1);
\draw (s) to node[pos=0.5,below]{} (u2);
\draw (s) to node[pos=0.5,below]{} (u3);
\draw (s) to node[pos=0.5,below]{} (u4);
\draw (s) to node[pos=0.5,below]{} (u5);
\draw (s) to node[pos=0.5,below]{} (u6);
\draw (s) to node[pos=0.5,below]{} (u7);
\draw (s) to node[pos=0.5,below]{} (u8);
\draw (s) to node[pos=0.5,below]{} (u9);
\draw (s) to node[pos=0.5,below]{} (u10);
\draw (s) to node[pos=0.5,below]{} (u11);
\end{tikzpicture}
}
\caption{An equitable (right) and an ordinary (left) coloring of a star. It is $\chi(G)=2$ and $\chi_{eq}(G)=7$, the color classes are indicated by the node labels.}
\label{fig:star}
\end{figure}
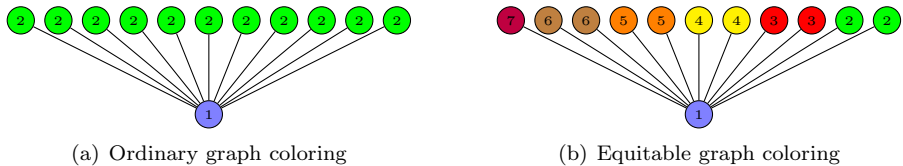

An equitable coloring of a star is presented in Figure~\ref{fig:star}. The example also shows that the difference between the chromatic number and the equitable chromatic number is, in principle, unbounded. In particular, given a star $S$ with $k\in\mathbb{N}_{\geq3}$ nodes, we have $\chi(S)=2$ and $\chi_{eq}(S)=\lceil\frac{k-1}{2}\rceil+1$.
\bigskip\\
ECP was introduced in 1973 by Meyer~\cite{M:1973}. For a summary of recent theoretical developments we refer to~\cite{KKMS:2010}. M\'{e}ndez-D\'{i}az et al. present a number of exact and heuristic algorithms for ECP in~\cite{MNS:2013,MNS:2014a,MNS:2014c,MNS:2014b}. In~\cite{MNS:2014b}, they describe an enumerative algorithm called EQDSATUR for ECP based on the DSATUR algorithm for ordinary graph coloring (see \cite{B:1979} for details), where they combine the branch and bound algorithm DSATUR with a customized pruning rule. In particular, they show that in many cases where a mixed integer linear program fails, such enumerative approach can be a reasonable alternative.
\bigskip\\
{\it Contribution/outline.} We focus on improvements to enumerative (DSATUR-based) algorithms to compute $\chi_{eq}(G)$, proposing a general scheme to derive pruning rules for these algorithms. 
We start with a detailed description of EQDSATUR as the
foundation of the following work. In Section~\ref{sec:haupt}, we show how the extendability of a partial coloring to an equitable coloring can be modeled via network flows. We remark that a similar construction has been made by de Werra~\cite{de1997restricted,de1999multiconstrained} and that this scheme also includes the pruning rule derived in~\cite{MNS:2014b}.

Computational experiments (Section~\ref{sec:compu}) show that a straight forward implementation where this scheme is evaluated by a flow problem
%
improves the original algorithm with respect to speed, number of solved instances given a fixed time limit and nodes in the search tree. 

However, the execution of flow algorithms at each node of a search tree is time consuming. Therefore, we discuss how arithmetical pruning rules can be derived from the network model. That is, we apply a result of Hoffman~\cite{hoffman1960} which yields a complete description of our network pruning model in generalized Hall-conditions. In Section~\ref{sec:hall}, we classify these conditions. We will see that there are exponentially many (non-dominated) conditions, a selection of which we apply within our algorithm. Computational Experiments show, that this yields an enumerative equitable coloring algorithm which is faster than the original algorithm in all but some of the largest instances.

\section{Definitions and Related Work}\label{sec:defs}

We repeat basic definitions and briefly describe the approach from  M\'{e}ndez-D\'{i}az et al.~\cite{MNS:2014b} to tackle ECP. For this section and the remainder of the work, we assume a graph $G=(V,E)$ to be given.  We start with some basic observations. 

Since any equitable coloring is a coloring, it is clear that $\chi(G)\leq \chi_{eq}(G)$. Let $k\in\mathbb{N}$ be fixed. If $G$ admits an equitable $k$-coloring, the sizes of the color classes are fixed.
\begin{Lemma}\label{sizeclass}
Let $k\in \mathbb{N}$, $n=|V|$ and  $p\equiv n \text{ mod } k $. If $G$ admits an equitable coloring with $k$ colors, then there are $p$ color classes of size $\left\lceil{\frac{n}{k}}\right\rceil$ and $k-p$ color classes of size 
$\left\lfloor{\frac{n}{k}}\right\rfloor.$ \hfill $\Box$
\end{Lemma}
Recall that a clique is a set of pairwise adjacent vertices and that we call a set of vertices stable if it only contains pairwise nonadjacent vertices.
In the following, we adopt several notations and definitions from \cite{MNS:2014b}. A partial $k$-coloring of $G$ is given by $\Pi_{p}:=(C_1,\dots,C_k,\dots,C_n)$ such that: all $C_i\subseteq V,i=1,\dots,k$ are pairwise disjoint, $C_i=\emptyset$ for $i\geq k+1$, and the $C_i$ are stable for $i=1,\dots,k$. Note, that a partial coloring can also be completely empty, i.e., $C_i=\emptyset$ for $i=1,\dots, n.$

We denote the set of uncolored vertices by $U(\Pi_p):=V\setminus (C_1\cup\dots\cup C_k)$. In the following, we make extensive use of the following property:
\begin{Definition}[extendability property]
 A partial $k$-coloring $\Pi_p$ can be {\it extended} to an (equitable) $\tilde{k}$-coloring for $\tilde{k}\geq k$ if and only if there exists an (equitable) $\tilde{k}$-coloring  
$\Pi_{eq}=(\tilde{C}_1,\dots,\tilde{C}_n)$ such that
 $\tilde{C}_i\supseteq C_i$ for $i=1,\dots,\tilde{k}.$ 
 \end{Definition}
 
For any colored vertex $v$ with $v\in C_i$ for some $i=1,\ldots,n$, we define $\Pi_p(v):=i$ as the $\Pi_p$-color of $v$. For any uncolored vertex $v\in U(\Pi_p)$, denote the forbidden colors of $v$ as 
$$D_{\Pi_p}(v):=\{\Pi_p(w)\mid w \in V\setminus U(\Pi_p), wv \in E \},$$  the set of free colors as $$F_{\Pi_p}(v):=\{1,\dots,n\}\setminus D_{\Pi_p}(v).$$
We refer to the saturation degree of vertex $v\in V$ as $\rho_{\Pi_p}(v):=|D_{\Pi_p}(v)|$. For any uncolored vertex $v\in U(\Pi_p)$ and a color $i\in F_{\Pi_p}(v)$, we define the operation $$\Pi_p+<v,i>:=(C_1,\dots,C_i\cup\{v\},\dots,C_n)$$ as the {\it extension} of $\Pi_p$ with $<v,i>$ which assigns the uncolored vertex $v$ to a (possibly empty) color class $C_i.$\\

An exact algorithm to solve ECP is EQDSATUR as presented in \cite{MNS:2014b}. Its basic idea is to enumerate all colorings of $G$ via DSATUR (cf. \cite{B:1979}) and, thereby, eliminating non minimal and non equitable colorings to find a minimum equitable coloring. The nodes of the enumeration tree correspond to (different) partial colorings which are subsequently extended to full colorings. 
\bigskip\\
EQDSATUR (see Algorithm~\ref{alg:eqdsatur}) is initialized at the root node of the enumeration tree with two global bounds for $\chi_{eq}$ (an upper bound $\overline{k}$, a lower bound $\underline{k}$) and with an initial partial coloring $\Pi_p$. For the sake of simplicity, these bounds are not denoted in Algorithm~\ref{alg:eqdsatur}. The initial partial coloring can, e.g., be empty or be a colored clique of $G$. Then, an uncolored vertex $v\in U(\Pi_p)$ of highest saturation degree, i.e., $\text{argmax}\{\rho_{\Pi_p}(v)\}$, is chosen. Note, that in case of a tie a random tie break (or more elaborate tie breaking rules cf. \cite{MNS:2014b}, e.g., PASS-VSS) can be applied at Step~\ref{choosev}. Now, for every available color $i\in F_{\Pi_p}(v),$ a new branch is generated, which is given by the extended partial coloring $\Pi_p+<v,i>.$ In this process, each branch inherits its own upper and lower bounds on $\chi_{eq}$. In Step~\ref{chooset}, the algorithm chooses, driven by a node selection criterion, the next node of the tree and continues as for the root node. If, at some point, a partial coloring is extended to a complete coloring, the branch is terminated. Then, it is checked whether the resulting coloring is equitable and yields an improved bound $\bar{k}$. Naturally, the algorithm stops when all nodes are pruned/terminated, i.e., when all potential partial colorings have been tested.
\bigskip\\
Note, that this basic form of DSATUR can naturally be intertwined with pruning rules, for instance in  Step~\ref{choosei}, when updating the bounds $\overline{k}$ and $\underbar{k} $ for the specific branches of the enumeration tree.

\begin{algorithm}
\begin{algorithmic}[1]
\Require{Graph $G=(V,E)$ with a partial coloring $\Pi_p$}
\Ensure $T\gets\left\{\Pi_p\right\}$ \Comment{$T\widehat{=}$ enumeration tree}
\While{$T\neq\emptyset$}
\State Select $\Pi_p\in T$ \label{chooset}
\State $T\gets T\setminus\left\{\Pi_p\right\}$
\If{$U(\Pi_p):=\emptyset$}
\State Evaluate $\Pi_p$.
\Else
\State Choose $v=\arg\max\left\{\rho_{\Pi_p}(u)\mid u\in U(\Pi_p)\right\}$ \label{choosev}
\For{$i \in F_{\Pi_p}(v)$}
\If{$not\;prune(\Pi_p+<v,i>)$}\label{choosei}
\State $T\gets T\cup \left\{\Pi_p+<v,i> \right\}$ \Comment{Store extended colorings}
\EndIf
\EndFor
\EndIf
\EndWhile
\end{algorithmic}
\caption{EQDSATUR}
\label{alg:eqdsatur}
\end{algorithm}

Due to the exponential size of such an enumeration, from a practical perspective, good pruning rules are essential to reduce the size of the tree. For ECP these rules can be categorized into two groups. Pruning if
 \begin{itemize}
  \item a known upper bound for $\chi_{eq}(G)$ cannot be improved by extending the coloring of the considered node.
  \item a partial coloring cannot be extended to an equitable coloring.
 \end{itemize}

We repeat some further definitions to introduce the results of \cite{MNS:2014b}. Consider $G$ and a partial coloring $\Pi_p,$ we denote by $M(\Pi_p):=\max\{|C_i| \mid i=1,\dots,n\}$ the size of the largest color class. Denote by $T(\Pi_p):=\{i\in\{1,\dots,n\}\mid |C_i|=M(\Pi_p)\}$ the indices of the largest color classes and the cardinality $t(\Pi_p):=|T(\Pi_p)|.$ 

\begin{Theorem}\cite{MNS:2014b}\label{firstprun}
 Consider $G$ and a partial $k$-coloring $\Pi_p.$ If this partial coloring can be extended to an equitable coloring, then
 \begin{eqnarray}  n\geq (M(\Pi_p)-1)(k-t(\Pi_p))+ M(\Pi_p)t(\Pi_p)= (M(\Pi_p)-1)k + t(\Pi_p).\label{prun}\end{eqnarray}
 Or, given a lower bound $\underbar{k}$ for $\chi_{eq}(G)$, it is $n \geq  (M(\Pi_p)-1) \max\{\underline{k},k\} + t(\Pi_p)$.
\end{Theorem}

\noindent The (first) condition of the foregoing theorem can be stated in a slightly different form. By subtracting $\sum_{i=1}^{k}|C_i|$ from both sides of the inequality, we obtain the equivalent formulation 
\begin{align}
|U(\Pi_p)|\geq \sum_{\substack{i=1 \\ |C_i|< M(\Pi_p)-1 }}^{k}(M(\Pi_p)-1-|C_i|).\label{easyprun}
\end{align}

EQDSATUR employs Theorem~\ref{firstprun} as pruning rule: At each node of the search tree, given a partial coloring $\Pi_p,$ an uncolored vertex $v\in U(\Pi_p)$ and a color $i\in F_{\Pi_p}(v)$, it is to check whether \begin{itemize} 
\item[(i)] $i\leq \overline{k}-1$ and 
\item[(ii)] $n\geq (M(\Pi_p')-1)\cdot \max\{\underline{k},k\} + t(\Pi_p')$ with $\Pi_p':= \Pi_p+<v,i>$. 
\end{itemize}
If any of the two conditions is not fulfilled, the branch is pruned. More informal,\vspace*{-0.3cm}\\

\begin{minipage}{0.45\textwidth}
the theorem translates to: ``There have to be enough uncolored vertices, to fill up all color classes to the size of the currently largest one minus one''. A visual interpretation of the result is given in Figure~\ref{fig:the1}. In the figure, a coloring is depicted in which the largest color class (blue) has size three, the other three color classes have size one. Since there
\end{minipage}
\hfill
\begin{minipage}{0.45\textwidth}
\begin{center}
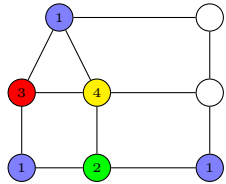

\begin{tikzpicture}[scale=0.5]
\node[draw, circle, fill = blue!50,inner sep=2pt] (p1) at (0,0) {\tiny 1};
\node[draw, circle, fill = green,inner sep=2pt] (p2) at (2,0) {\tiny 2};
\node[draw, circle, fill = red,inner sep=2pt] (p3) at (0,2) {\tiny 3};
\node[draw, circle, fill = yellow,inner sep=2pt] (p4) at (2,2) {\tiny 4};
\node[draw, circle, fill = blue!50,inner sep=2pt] (p5) at (1,4) {\tiny 1};
\node[draw, circle, fill = blue!50,inner sep=2pt] (p6) at (5,0) {\tiny 1};
\node[draw, circle,inner sep=2pt] (p7) at (5,2) {\textcolor{white}{\tiny 2}};
\node[draw, circle,inner sep=2pt] (p8) at (5,4) {\textcolor{white}{\tiny 3}};

\draw (p1) -- (p2);
\draw (p2) -- (p4);
\draw (p4) -- (p3);
\draw (p3) -- (p1);

\draw (p3) -- (p5);
\draw (p4) -- (p5);

\draw (p6) -- (p2);
\draw (p7) -- (p4);
\draw (p8) -- (p5);

\draw (p6) --(p7);
\draw (p8) --(p7);
\end{tikzpicture} 
\captionof{figure}{A non-extendable coloring}
\label{fig:the1}
\end{center}
\end{minipage}\vspace*{0.125cm}\\
are only two uncolored nodes left, the coloring cannot be extended because not all color classes can be filled up to include at least two nodes each (which requires at least three additional nodes).

\section{Flows and Partial Colorings} \label{sec:haupt}

In this section, we present a flow-based scheme to model necessary conditions on partial colorings to be extendable. By means of this scheme Theorem~\ref{firstprun} and additional, stronger conditions for extendability can be derived. Note that the presented construction is similar to a network model presented in~\cite{de1997restricted,de1999multiconstrained}. Throughout this section, we assume a graph $G$ together with a partial coloring $\Pi_p$ with $k$ colors to be given. Therefore, we consider the simplified notation 
$D(v):=D_{\Pi_p}(v), F(v):=F_{\Pi_p}(v), M:=M(\Pi_p),\rho(v):=\rho_{\Pi_p}(v)$ and $U:=U(\Pi_p).$

\subsection{Modeling Extendability via Network Flows }
We model the extendability of a partial coloring $\Pi_p$ to an equitable $k_0$ coloring by a flow. That is, we describe a network where the vertices correspond to (uncolored) nodes of $G$ and to the $k_0$ available colors, respectively to their color classes. The coloring is extendable if a flow from each uncolored vertex to a color (class) exists.

Let a number of colors $k_0\leq n$ be given. We emphasize that $k_0$ is not necessarily equal to $k$ as the number of colors used in $\Pi_p$. Let $U$ be the union of $l$ disjoint vertex sets $U^1,\dots,U^l$. 
Moreover, denote by $\alpha^j=\alpha(G\left[U^j\right])$ an upper bound on the size of the maximum stable set, i.e., on the stability number, of the subgraph induced by $U^j$.
In order to model the extendability of $\Pi_p$ we define the directed network $N(G,\Pi_p,k_0):=(V_N,A_N)$. For a visualization, we refer 
to Figure~\ref{network}. 

Formally, the construction writes as follows: Let $C=\{1,\ldots,k_0\}$ be a set of vertices corresponding to the considered colors. With a slight abuse of notation, let $U=\bigcup_{j=1}^{l}U^j$ be a set of vertices corresponding to the uncolored nodes in $G$. For each $j=1,\ldots,l$, let $F^j$ be a copy of $C$ and let $F=\bigcup_{j=1}^{l}F^j$. We write $f^j(i)$ for the vertex in $F^j$ corresponding to (a copy of) color $i$. Then, adding a source node $s$ and a sink node $t$, the vertex set $V_N$ is defined as
\begin{align*}
V_N:=\{s\} \cup \bigcup_{j=1}^{l} U^j   \cup \bigcup_{j=1}^{l} F^j  \cup  C \cup \{t\}. 
\end{align*}
The arc set $A_N$ is then composed of the arc sets $A_1,\ldots A_4$. Hereby, $A_1$ connects $s$ to $U$, i.e., $A_1:=\{(s,v) \mid v\in U\}$ and  $A_2$ connects $U$ to $F$, that is
\begin{align*}
A_2:= \bigcup_{j=1}^{l} \left\{(v,f^j(i))\mid \ v\in U_j  \text{ and } i\in F(v) \right\}.
\end{align*}
Similar, $A_3$, connects each node in $F$ to the corresponding color in $C$
\begin{align*}
A_3:= \bigcup_{j=1}^{l} \left\{(f^j(i),i)\mid \ i=1,\dots,k_0 \right\}
\end{align*}
and $A_4$ connects $C$ to $t$, that is $A_4:=\{(i,t)\mid  i=1,\dots,k_0\}$.
%
%

\begin{figure}
\begin{center}
\begin{tikzpicture}[scale=0.3125,>=latex]
\fill[color=blue!10] (0,-10) rectangle (6,15); 
\path (3,13.25) node[above] {$A_1$};
\path (3,11.5) node[above]  {{\tiny $\begin{array}{l} U\!\!=\!\!1 \\ L\!\!=\!\!0\end{array}$}};
\fill[color=blue!5] (6,-10) rectangle (12,15);
\path (9,13.25) node[above] {$A_2$};
\path (9,11.5) node[above]  {{\tiny $\begin{array}{l} U\!\!=\!\!1 \\ L\!\!=\!\!0\end{array}$}};
\fill[color=blue!10] (12,-10) rectangle (18,15); 
\path (15,13.25) node[above] {$A_3$};
\path (15,11.5) node[above]  {{\tiny $\begin{array}{l} U\!\!=\!\!a^j \\ L=\!\!0\!\!\end{array}$}};
\fill[color=blue!5] (18,-10) rectangle (24,15);
\path (21,13.25) node[above] {$A_4$};
\path (21,11) node[above]  {{\tiny $\begin{array}{l} U\!\!=\!\!\lceil\frac{n}{k_0}\rceil\!\!-\!\!|C_i| \\ L\!\!=\!\!\lfloor\frac{n}{k_0}\rfloor\!\!-\!\!|C_i|\end{array}$}};
\draw (6,8) ellipse (1 and 2.5); 
\path (7,10) node[above] {$U^1$}; 
\draw (6,2) ellipse (1 and 2.5);
\path (7,4) node[above] {$U^2$};
\draw (6,-7) ellipse (1 and 2.5);
\path (7,-5) node[above] {$U^l$};
\draw(12,8) ellipse (1 and 2.5);
\draw(12,2.5) ellipse (1 and 2.5);
\draw(12,-7) ellipse (1 and 2.5);
\draw(18,1) ellipse (1 and 2.5);
\path (10.75,10) node[above] {$F^1$}; 
\path (10.75,4) node[above] {$F^2$};
\path (10.75,-5) node[above] {$F^l$};
\path (18,3.5) node[above] {$C$};
{\tikzstyle{every node}=[draw,shape=circle, inner sep=0pt,minimum size=6pt];   

\node[inner sep=0pt,minimum size=10pt,fill=gray!50](s) at (0,0) {$s$};

\node(u11) at (6,10) {};
\node(u12) at (6,9) {};
\node(u13) at (6,6) {};
\node(u21) at (6,4) {};
\node(u22) at (6,3) {};
\node(u23) at (6,0) {};
\node(u31) at (6,-5) {};
\node(u32) at (6,-6) {};
\node(u33) at (6,-9) {};

\node[fill=blue!50](v11) at (12,9.5) {};
\node[fill=green](v12) at (12,8.5) {};
\node[fill=purple](v13) at (12,6.25) {};
\node[fill=blue!50](v21) at (12,4) {};
\node[fill=green](v22) at (12,3) {};
\node[fill=purple](v23) at (12,0.75) {};
\node[fill=blue!50](v31) at (12,-5.5) {};
\node[fill=green](v32) at (12,-6.5) {};
\node[fill=purple](v33) at (12,-8.75) {};

\node[fill=blue!50](w1) at (18,3) {};
\node[fill=green](w2) at (18,2) {};
\node[fill=purple](w3) at (18,-1) {};

\node[inner sep=0pt,minimum size=10pt,fill=gray!50](t) at (24,0) {$t$};

}
\draw (s) to (u11);
\draw (s) to (u12);
\draw (s) to (u13);
\draw (s) to (u21);
\draw (s) to (u22);
\draw (s) to (u23);
\draw (s) to (u31);
\draw (s) to (u32);
\draw (s) to (u33);
\draw[thick, decorate,decoration={zigzag,segment length=7,amplitude=2}] 
(7,8) to node[pos=0.5,below]{} (11,8);
\draw[thick, decorate,decoration={zigzag,segment length=7,amplitude=2}] 
(7,3) to node[pos=0.5,below]{} (11,3);
\draw[thick, decorate,decoration={zigzag,segment length=7,amplitude=2}] 
(7,-7) to node[pos=0.5,below]{} (11,-7);
\path (9,-1) node {\tiny arc $(u,f^j(i))$};
\path (9,-2) node {\tiny exists};
\path (9,-3) node {\tiny iff $i\in F(u).$ };
\draw (8.8,-0.5) to (8,3);
\draw (9.2,-0.5) to (10,8);
\draw (9,-3.5) to (9,-7);

\draw (v11) to (w1);
\draw (v12) to (w2);
\draw (v13) to (w3);

\draw (v21) to (w1);
\draw (v22) to (w2);
\draw (v23) to (w3);

\draw (v31) to (w1);
\draw (v32) to (w2);
\draw (v33) to (w3);

\draw (w1) to (t);
\draw (w2) to (t);
\draw (w3) to (t);

\path (6,7.75) node {\bf $\vdots$};
\path (6,1.75) node {\bf $\vdots$};
\path (6,-1.75) node {\bf $\vdots$};
\path (6,-7.25) node {\bf $\vdots$};

\path (12,7.7) node {\bf $\vdots$};
\path (12,2.2) node {\bf $\vdots$};
\path (12,-1.75) node {\bf $\vdots$};
\path (12,-7.3) node {\bf $\vdots$};

\path (18,0.5) node {\bf $\vdots$};

\end{tikzpicture} 
 \caption{The network $N(G,\Pi_p,k):=(V_N,A_N)$. The ``uncolored'' vertices are depicted in white, the vertices corresponding to colors are depicted in their respective colors. $L$ and $U$ indicate the arc capacities.}
\label{network}
\end{center}
\end{figure}
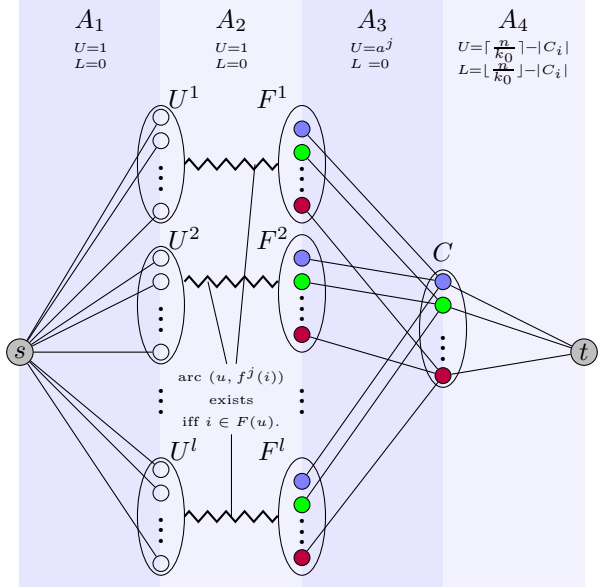
We define the corresponding arc capacities as follows: Let minimum/maximum arc capacities be $c:A_N\mapsto \mathbb{N}_0$ and $d:A_N\mapsto \mathbb{N}_0$ with 
$$\arraycolsep=1pt \begin{array}{lcll}
 c(a) :=  \lfloor{\frac{n}{k_0}}\rfloor - |C_i| & \text{ and } & d(a):= \lceil{\frac{n}{k_0}}\rceil - |C_i|, & \text{if } a=(i,t)\in A_4, \\
 c(a) := 0 & \text{ and } & d(a):=\alpha^j, & \text{if } a=(f^j(i),i) \in A_3,\\
 c(a) := 0 & \text{ and } & d(a):=1, &  \text{for all other edges} \\
\end{array}$$
for $i=1,\dots,k_0$ and for $j=1,\dots,l$. 

\begin{Theorem}\label{main}
Consider $G$ and its partial $k$-coloring $\Pi_p$. Assume $k_0\geq k$ and that $M\leq \lceil{\frac{n}{k_0}}\rceil.$ Suppose that the set of uncolored vertices $U$ is the union of $l$ disjoint vertex sets $U^1,\dots, U^l.$ Let $\alpha^j$ be given for $j=1,\dots,l$.\\ If $\Pi_p$ can be extended to an equitable $k_0$-coloring, then
 the network $N(G,\Pi_p,k_0)$ has an admissible flow of value $|U|$.
\end{Theorem}
\begin{proof} 
The transformation of an extension $\Pi$ of $\Pi_p$ into an admissible flow $x$ of value $|U|$ in $N(G,\Pi_p,k)$ is easy (compare Table \ref{tab:flowcolor}).
\begin{table}
\renewcommand{\arraystretch}{0.6}
 \begin{center}
\setlength{\tabcolsep}{8pt}
 \caption{An equitable coloring $\Pi$ transformed into an admissible flow $x.$}

\label{tab:flowcolor}
 \begin{tabular}{clclc}
\toprule
 Arc & \multicolumn{4}{c}{ Flow value } \\ \midrule
 $A_1$ & $x_{(s,v)}$  & $=$ & $1,$ &  \\ 
 $A_2$ & $x_{(v,f^j(i))}$  &=& $1,$ & if $\Pi(v)=i$ \\
       & $x_{(v,f^j(i))}$ &=& $0,$ & if $\Pi(v)\neq i$ \\
 $A_3$ & $x_{(f^j(i),i)}$ & $=$ & $ |\{v\in U^j\mid \Pi(v)=i\}|,$ & \\
 $A_4$ & $x_{(i,t)}$ & $=$ & $\sum\limits_{j=1}^{l}x_{(f^j(i),i)}.$ & \\
\bottomrule
 \end{tabular}
 \end{center}

\end{table}
It is clear that the flow obeys the capacity restrictions for the arcs in $A_1$. 
We check the capacities on $A_2, A_3$ and $A_4:$\\
Since $\Pi$ extends $\Pi_p,$  a node $v\in U$ can only be colored with colors from $F(v)$. $A_2$ contains only one arc connecting $v\in U^j$ to $f^j(i)$ for each color $i \in F(v).$ Thus, the flow on $A_2$ is well defined and obeys the capacity constraints.  

By construction, all of the nodes $f^j(i)$ are connected to the node corresponding to the color $i\in C$. Hence, the flow on $(f^j(i),i)\in A_3$ is well defined. Since in any (equitable) coloring, each color class yields a stable set, the flow $|\{v\in U^j\mid \Pi(v)=i\}|$ on the arc $(f^j(i),i),$ obeys to the given capacity, which is chosen as upper bound on the stability number.

For any arc $(i,t)$ of $A_4$, Lemma \ref{sizeclass} gives us the size of the corresponding color classes $C_i$ of $\Pi,$ namely $\lfloor{\frac{n}{k_0}}\rfloor$ or $\lceil{\frac{n}{k_0}}\rceil.$ Subtracting the nodes which are already colored by $\Pi_p,$ we obtain $\lfloor{\frac{n}{k_0}} \rfloor - |C_i|$ or $\lceil{\frac{n}{k_0}}\rceil - |C_i|.$ Hence, $\Pi$ assigns either $\lfloor{\frac{n}{k_0}}\rfloor - |C_i| $ or $\lceil{\frac{n}{k_0}}\rceil - |C_i|$ vertices from $U$ to $i.$
Therefore,  $$\sum\limits_{j=1}^{l}x_{(f^j(i),i)} = \sum\limits_{j=1}^{l} |\{v\in U^j\mid \Pi(v)=i\}| \in  \left\{\left\lfloor{\frac{n}{k_0}}\right\rfloor - |C_i|,  \left\lceil{\frac{n}{k_0}}\right\rceil - |C_i| \right\}$$ and all capacity constraints hold. Flow balance holds by construction and the flow given by Table~\ref{tab:flowcolor} has value $|U|$. Since $\sum_{v\in U}c((s,v))=|U|$, the flow is maximum. 
\end{proof}

The condition $\lceil{\frac{n}{k_0}}\rceil\geq M$ is worth explaining: If $M$ (the size of a largest color class) is greater than  $\lceil{\frac{n}{k_0}}\rceil,$ there is no equitable $k_0$-coloring extending $\Pi_p,$ because we need color classes of size at most  $\lceil{\frac{n}{k_0}}\rceil.$

Theorem~\ref{main} establishes a necessary criterion when a partial $k$-coloring is extendable to an equitable $k_0$-coloring for fixed $k_0\geq k$ and thus, it can be extended to serve as a pruning rule within EQDSATUR:

Given a partial coloring $\Pi_p$ in the search tree of EQDSATUR and the collection of uncolored vertices $U,$ then $U$ is decomposed into sets $U^j$ and the condition is to be tested for any $k\leq k_0\leq \max\{\tilde k\in \mathbb{N}\mid M\leq \lceil{\frac{n}{\tilde k}}\rceil, \tilde k\leq \bar{k}\}$. If none of the flow problems yields the desired flow, the current branch is pruned. For later reference, we rephrase the pruning scheme in Algorithm~\ref{alg:prune}. If the algorithm returns true, the current branch in the EQDSATUR can be pruned.
\begin{algorithm}
\begin{algorithmic}[1]
\Require{Partial Coloring $\Pi_p$, upper bound $\bar k \geq \chi_{eq}$}
\Ensure bool $p \gets true$
\State Determine decomposition $U=\bigcup_{j=1}^{l}U^j$ \label{chooseDecomp}
\State Determine bounds $\alpha^j\quad\forall j=1,\ldots, l$
\For{$k_0 \in \{k\ldots,\max\{\tilde k\in \mathbb{N}\mid M\leq \lceil{\frac{n}{\tilde k}}\rceil, \tilde k\leq \bar{k}\}\}$}
\If{$N(G,\Pi_p,k)$ has an admissible flow of value $|U|$} \label{solveflow}
\State $p \gets false$
\State break
\EndIf
\EndFor
\State \textbf{Return} p
\end{algorithmic}
\caption{prune($\Pi_p)$}
\label{alg:prune}
\end{algorithm}

In this context, the choice of the decomposition of $U$ (compare Algorithm~\ref{alg:prune}, step~\eqref{chooseDecomp}) is a central element for the success of the pruning scheme. However, it is not clear what a `best' decomposition of $U$ for the ECP is. Still, given a decomposition into sets $U^j$, the knowledge of good bounds $\alpha^j$ is crucial from an algorithmic point of view. The discussion of different decompositions is the topic of the next subsection. 
 

\subsection{Decomposing the set of uncolored vertices}\label{choices}

As as first step, we establish a relation between the choices of $U^j$ and $\alpha^j$, and the result stated in Theorem~\ref{prun}.  
After that, we discuss a setting, for which Theorem~\ref{main} yields a characterization of extendability. In this case, \textit{if and only if} Algorithm~\ref{alg:prune} returns true, the current branch in the EQDSATUR can be pruned. Finally, the previous findings are combined to achieve a strong setting for practical use. 

\subsubsection{A direct and fast approach} \label{weak}
We consider the trivial decomposition $U^1:=U$ and let $\alpha^1:=|U|$. 
Then, for selected $k_0,$ the network $N(G,\Pi_p,k_0)$ boils down to
$\tilde{N}(G,\Pi_p,k_0) $ with
\begin{align*}
&V_{\tilde N}:= \{s\}\cup U \cup C \cup \{t\}\quad\textnormal{and}\\
&A_{\tilde N}:= A_1 \cup \{(v,i)\mid \ v\in U  \text{ and } i\in F(v) \} \cup A_4.
\end{align*}
%
and flow capacities $c:A_{\tilde{N}}\mapsto \mathbb{N}_0$ and $d:A_{\tilde{N}}\mapsto \mathbb{N}_0$ with 
$$\arraycolsep=1pt \begin{array}{lcll}
 c(a) :=  \lfloor{\frac{n}{k_0}}\rfloor - |C_i| & \text{ and } & d(a):= \lceil{\frac{n}{k_0}}\rceil - |C_i|, & \text{for } a=(i,t), \\
 c(a) := 0 & \text{ and } & d(a):=1, &  \text{for all other edges.} \\
\end{array}.$$

Since the capacities of arcs in $A_3$ are $|U|=|U^1|,$ they do not restrict the flow any more and can be omitted. 

As a further relaxation, assume that $F(v)=C$ for all $v\in U$. In other words, we assume that any vertex can still be colored with all colors. For a potential application of Theorem~\ref{main}, we are looking for a maximum flow of value $|U|$ in the corresponding network $\tilde{N}$. A necessary condition for the existence of such flow is
$$|U|\geq \sum_{\substack{i=1\\|C_i|<M}}^{k}(M-1-|C_i|),$$ 
Implying that there are enough uncolored vertices to reach equitability, resp. to fill up the color classes to the size of the largest one minus $1$.

Assume the contrary. Consider the cut induced in $\tilde{N}$ by $V\setminus \left\{t\right\}$. There, the overall required (minimum) flow is 
\begin{align*}
 \sum_{(i,t)\in A_4}\left(\left\lfloor \frac{n}{k_0}\right\rfloor - |C_i|\right)\geq \sum_{i=1}^{k_0}\left(\left\lceil \frac{n}{k_0}\right\rceil - |C_i|-1\right)\geq \sum_{\substack{i=1\\|C_i|<M}}^k \left( M- |C_i|-1\right),
\end{align*}
which can not be fulfilled if less flow leaves the source. So, we obtain the condition from~\cite{MNS:2014b} for extendability (compare Theorem~\ref{firstprun}). 
This condition offers two practical advantages: It is not depending on $k_0$ and it is easy to verify since it does not require to solve any flow problem.

\subsubsection{Strong decompositions} \label{strong}

We consider the case that $U$ can be decomposed into $l$ pairwise nonadjacent cliques $U^j$ and hence, it is $\alpha^j=1$ for all $j=1,\ldots,l$. 
Then, for fixed $k_0,$ the flow model $N(G,\Pi_p,k_0)$ even yields a characterization of extendability:
\begin{Theorem}\label{exact}
Let $k_0\geq k$ and $M\leq \lceil{\frac{n}{k_0}}\rceil.$ Suppose that the set $U$ decomposes into pairwise non-adjacent cliques $U^1,\dots,U^l$. The partial coloring can be extended to an equitable $k_0$-coloring if and only if the network $N(G,\Pi_p,k_0)$ has an admissible flow of value $|U|.$
\end{Theorem}
\begin{proof}
 Suppose that 
$N$ has a maximum flow $x$ of size $|U|$. The capacities 
\begin{align*}
 c(a)= 0 \leq x_{(f^j(i),i)} \leq d(a)=\alpha^j=1 
\end{align*}
for all arcs $a\in A_3$ and for all $i=1,\ldots,k_0$, $j=1,\ldots,l$ ensure that each vertex $v\in U^j$ gets a different color $i\in F(v)$.
Hence, the flow transforms into a proper equitable $k_0$-coloring $\Pi$ of $G$, namely 
\begin{align*}
 \Pi=\Pi_p+\sum\limits_{i=1}^{k_0}\sum\limits_{j=1}^{l}\sum_{\substack{ v\in U^j: \\  x_{(v,f^j(i) )}=1 }} <v,i>. 
\end{align*}

If the network $N$ has no admissible flow of value $|U|,$ Theorem~\ref{main} shows that there is no equitable $k_0$-coloring which extends $\Pi_p.$
\end{proof}

In the case where $U$ decomposes into non adjacent cliques, the scheme corresponds to the model employed by de Werra in~\cite{de1997restricted,de1999multiconstrained} applied to equitable coloring.

\subsubsection{A mixed approach}\label{mixed} 

Up to now we have seen two decomposition variants of $U.$ It is clear that the first one yields a very practicable but weak approach, while the second one provides a very impractical (since in general $U$ will not decompose into non-adjacent cliques) but strong approach with respect to pruning rules for EQDSATUR. Therefore, we propose a heuristic combination of the two  approaches.
\bigskip\\
We decompose the set $U$ into non-adjacent cliques $U^1,\dots,U^j$ and some remaining vertices $U^0$. The idea is that the part of the auxiliary network $N$ containing the cliques will provide a relatively strong pruning condition, while the remaining part can only contribute to a weak bound. 
With respect to Theorem~\ref{main}, we set $\alpha^{j}:=1$ for $j=1,\ldots,l$ and $\alpha^{0}:=|U^0|$ for the remaining noes. In this setting, the nodes in $U^1,\dots,U^l$ can have connections to nodes in $U^0$, such that $N$ does not give a sufficient but a necessary condition for the extension property. However, in the course of EQDSATUR, $|U|$ decreases. If at some point, $U^0=\emptyset$, $N$ collapses to the strong setting and yields a characterization of extendability at which the current branch can be terminated at latest.
In this decomposition, only cliques $U^j$ with $|U^j|\geq2$ have to be considered because singletons can be put into $U^0$ without weakening the pruning.
\begin{Example}\label{ex:mixed}
An example of the mixed approach is given in Figure~\ref{fig:robertgraph}. In this figure, we consider a graph with $12$ nodes, which we color with a DSATUR algorithm. An equitable coloring with four colors exists and we assume that four is a known upper bound at the beginning of the algorithm. Therefore, the algorithm searches for a better coloring, e.g., one that utilizes only three colors. The algorithm colors the node in sequence, starting with one, two and so forth. As soon as a color is selected for node one, Theorem~\ref{main} diretly implies that the partial coloring cannot be extended. Hence, the algorithm immediatly terminates. This is easy to see, as each color class would contain exactly four vertices, two of which are required to be taken by vertices within the two cliques. Assuming that node one is in color class $C_{i}$ for some color $i$, all remaining nodes are connected to node one, i.e., the color class $C_{i}$ cannot get more than three nodes. This is a contradiction.\\
Solely based on Theorem~\ref{firstprun}, the non-extendability could not have been verified as early. An algorithm with just Theorem~\ref{firstprun} as pruning rule needs to visit hundreds of partial colorings.
\end{Example}
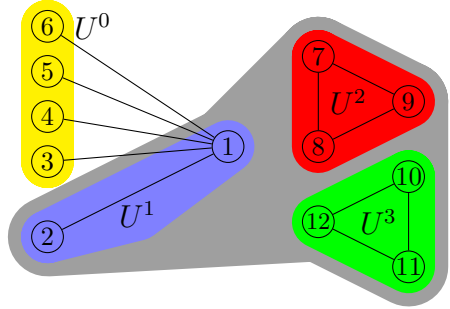
\begin{figure}

\begin{center}
\begin{tikzpicture}[scale=0.4,>=latex]
\tikzstyle{vertex} = [fill,shape=circle,node distance=40pt]
\tikzstyle{edge} = [ opacity=.5, line cap=round, line join=round, line width=25pt]

{\tikzstyle{every node}=[draw, shape=circle, inner sep=0pt,minimum size=12pt];   
\path (0,0) node (s){1};}
{\tikzstyle{every node}=[draw, shape=circle, inner sep=0pt,minimum size=12pt];   
\path (-6,-3) node (u1){};
\path (-6,-0.5) node (u2){};
\path (-6,1) node (u3){};
\path (-6,2.5) node (u4){};
\path (-6,4) node (u5){};

\path (3,3) node (c1){};
\path (3,0) node (c2){};
\path (6,1.5) node (c3){};

\path (6,-1) node (d1){};
\path (6,-4) node (d2){};
\path (3,-2.5) node (d3){};

}

\begin{scope}[transparency group,opacity=.5]
\fill[edge,opacity=1,color=yellow] (u2.center) -- (u3.center) -- (u4.center) -- (u5.center) -- (u2.center);
\draw[edge,opacity=1,color=yellow,line width=20pt] (u2.center) -- (u3.center) -- (u4.center) -- (u5.center) -- (u2.center);
\end{scope}

\begin{scope}[transparency group,opacity=.5]
\fill[edge,opacity=1,color=gray!75] (c3.center) -- (c1.center) -- (c2.center) -- (c3.center) -- (d1.center) -- (d2.center) -- (d3.center) --  (u1.center) -- (s.center) -- (c1.center);
\draw[edge,opacity=1,color=gray!75,line width=30pt](c3.center) -- (c1.center) -- (c2.center) -- (c3.center) -- (d1.center) -- (d2.center) -- (d3.center) --  (u1.center) -- (s.center) -- (c1.center);
\end{scope}

\begin{scope}[transparency group,opacity=.8]
\fill[edge,opacity=1,color=blue!50] (u1.center) -- (s.center) -- (-3,-2.25) --  (u1.center);
\draw[edge,opacity=1,color=blue!50,line width=20pt] (u1.center) -- (s.center) -- (-3,-2.25) -- (u1.center) ;
\end{scope}

\begin{scope}[transparency group,opacity=.8]
\fill[edge,opacity=1,color=red] (c1.center) -- (c2.center) -- (c3.center) -- (c1.center);
\draw[edge,opacity=1,color=red,line width=20pt] (c1.center) -- (c2.center) -- (c3.center) -- (c1.center) ;
\end{scope}

\begin{scope}[transparency group,opacity=.8]
\fill[edge,opacity=1,color=green] (d1.center) -- (d2.center) -- (d3.center) -- (d1.center);
\draw[edge,opacity=1,color=green,line width=20pt] (d1.center) -- (d2.center) -- (d3.center) -- (d1.center) ;
\end{scope}


\draw (s) to node[pos=0.5,below]{} (u1);
\draw (s) to node[pos=0.5,below]{} (u2);
\draw (s) to node[pos=0.5,below]{} (u3);
\draw (s) to node[pos=0.5,below]{} (u4);
\draw (s) to node[pos=0.5,below]{} (u5);
\draw (c1) to node[pos=0.5,below]{} (c2);
\draw (c2) to node[pos=0.5,below]{} (c3);
\draw (c3) to node[pos=0.5,below]{} (c1);
\draw (d1) to node[pos=0.5,below]{} (d2);
\draw (d2) to node[pos=0.5,below]{} (d3);
\draw (d3) to node[pos=0.5,below]{} (d1);

\path (-3,-2.25) node {\large $U^1$};
\path (4,1.5) node {\large $U^2$};
\path (5,-2.5) node {\large $U^3$};
\path (-4.5,4) node {\large $U^0$};

{\tikzstyle{every node}=[draw, shape=circle, inner sep=0pt,minimum size=12pt];   
\path (0,0) node (s){1};}
{\tikzstyle{every node}=[draw, shape=circle, inner sep=0pt,minimum size=12pt];   
\path (-6,-3) node (u1){$2$};
\path (-6,-0.5) node (u2){$3$};
\path (-6,1) node (u3){$4$};
\path (-6,2.5) node (u4){$5$};
\path (-6,4) node (u5){$6$};

\path (3,3) node (c1){$7$};
\path (3,0) node (c2){$8$};
\path (6,1.5) node (c3){$9$};

\path (6,-1) node (d1){$10$};
\path (6,-4) node (d2){$11$};
\path (3,-2.5) node (d3){\small $12$};
}
\end{tikzpicture} 

\caption{A visualization of Example~\ref{ex:mixed} of the mixed approach. The set of non-adjacent cliques (blue, red green) is indicated by the gray area. The yellow nodes are not contained in any clique.}
\label{fig:robertgraph}
\end{center}
\end{figure}

Again, we remark that a `best' decomposition of $U$ is not known, but intuitively, it is desirable to cover as many vertices as possible by the cliques. 
Such a decomposition can always be obtained greedily, e.g., as it is shown in Algorithm~\ref{alg:nonadjcliques}. Note that it requires a few additional definitions: For a subset $U\subseteq V$ resp. a vertex $v\in V$ we denote by $\delta(U):=\{w\in V\setminus U \mid vw\in E \text{ for some } v\in U\}$ and by $\delta(v):=\{w\in V\setminus \{v\} \mid vw\in E \}$ the open neighborhood of $U$ resp. of $v$. 
We denote the degree of a vertex $v\in V$ by $\deg_G(v).$

\begin{algorithm}
\caption{Find Non Adjacent Cliques}\label{alg:nonadjcliques}
\begin{algorithmic}[1]
\Require{Graph $G=(V,E)$ with a partial coloring $\Pi_p$}
\Ensure $i\gets1$, $U^0\gets \emptyset$, $U^i\gets \emptyset$
\While{$U\neq\emptyset$}
\State $v\gets\underset{v\in U}{argmax}\left\{\deg_G(v)\right\}$
\State $U^i\gets U^i\cup\left\{v\right\}$ \Comment{construct $U^i$ as clique}
\While{$\exists v \in \bigcap_{w\in U^i}\delta(w)$}
\State $v\gets\underset{v\in \bigcap_{w\in U^i}}{argmax}\left\{\deg_G(v)\right\}$
\State $U^i\gets U^i\cup\left\{v\right\}$
\EndWhile
\State $U\gets U\setminus\left(U^i\cup \delta(U^i)\right)$
\State $U^0\gets U^0\cup \delta(U^i)$ \Comment{neighbors of $U^i$ go to $U^0$}
\State $i \gets i+1$, $U^i\gets \emptyset$
\EndWhile
\end{algorithmic}
\end{algorithm}

We point out that at the very first iteration, with a slight increase in computation time, the algorithm can be executed for different starting nodes to obtain a better result. 

\section{Extendability and Hall Conditions}\label{sec:hall}

\begin{figure}
\begin{center}
\begin{tikzpicture}[scale=0.3125,>=latex]
\fill[color=blue!5] (6,-10) rectangle (12,14);
\path (9,12) node[above] {$A_2$};
\fill[color=blue!10] (12,-10) rectangle (18,14); 
\path (15,12) node[above] {$A_3$};
\draw (6,8) ellipse (1 and 2.5); 
\path (7,10) node[above] {$U^1$}; 
\draw (6,2) ellipse (1 and 2.5);
\path (7,4) node[above] {$U^2$};
\draw (6,-7) ellipse (1 and 2.5);
\path (7,-5) node[above] {$U^0$};
\draw(12,8) ellipse (1 and 2.5);
\draw(12,2.5) ellipse (1 and 2.5);
\draw(12,-7) ellipse (1 and 2.5);
\draw(18,1) ellipse (1 and 2.5);
\path (10.75,10) node[above] {$F^1$}; 
\path (10.75,4) node[above] {$F^2$};
\path (10.75,-5) node[above] {$F^0$};
\path (18,3.5) node[above] {$C$};
\draw [decorate,decoration={brace,amplitude=10pt},xshift=-4pt,yshift=0pt]
(5,-3.125) -- (5,10.5) node [black,midway,xshift=-1.25cm] 
{cliques $U^j$};
\draw[gray, dashed, thick] (0,-3.35) -- (22,-3.35);

{\tikzstyle{every node}=[draw,shape=circle, inner sep=0pt,minimum size=6pt];   


\node(u11) at (6,10) {};
\node(u12) at (6,9) {};
\node(u13) at (6,6) {};
\node(u21) at (6,4) {};
\node(u22) at (6,3) {};
\node(u23) at (6,0) {};
\node(u31) at (6,-5) {};
\node(u32) at (6,-6) {};
\node(u33) at (6,-9) {};

\node[fill=blue!50](v11) at (12,9.5) {};
\node[fill=green](v12) at (12,8.5) {};
\node[fill=purple](v13) at (12,6.25) {};
\node[fill=blue!50](v21) at (12,4) {};
\node[fill=green](v22) at (12,3) {};
\node[fill=purple](v23) at (12,0.75) {};
\node[fill=blue!50](v31) at (12,-5.5) {};
\node[fill=green](v32) at (12,-6.5) {};
\node[fill=purple](v33) at (12,-8.75) {};

\node[fill=blue!50](w1) at (18,3) {};
\node[fill=green](w2) at (18,2) {};
\node[fill=purple](w3) at (18,-1) {};


}
\draw[thick, decorate,decoration={zigzag,segment length=7,amplitude=2}] 
(7,8) to node[pos=0.5,below]{} (11,8);
\draw[thick, decorate,decoration={zigzag,segment length=7,amplitude=2}] 
(7,3) to node[pos=0.5,below]{} (11,3);
\draw[thick, decorate,decoration={zigzag,segment length=7,amplitude=2}] 
(7,-7) to node[pos=0.5,below]{} (11,-7);

\draw (v11) to (w1);
\draw (v12) to (w2);
\draw (v13) to (w3);

\draw (v21) to (w1);
\draw (v22) to (w2);
\draw (v23) to (w3);

\draw (v31) to (w1);
\draw (v32) to (w2);
\draw (v33) to (w3);



\path (6,7.75) node {\bf $\vdots$};
\path (6,1.75) node {\bf $\vdots$};
\path (6,-1.75) node {\bf $\vdots$};
\path (6,-7.25) node {\bf $\vdots$};

\path (12,7.7) node {\bf $\vdots$};
\path (12,2.2) node {\bf $\vdots$};
\path (12,-1.75) node {\bf $\vdots$};
\path (12,-7.3) node {\bf $\vdots$};


\path (18,0.5) node {\bf $\vdots$};


\end{tikzpicture} 
\caption{(Modified) Structure of the network $N(G,\Pi,k_0):=(V_N,A_N)$. The source $s$ and the target $t$ have been omitted. Note that the superscript $0$ refers to nodes not corresponding to a clique.}
\label{fig:modflow}
\end{center}
\end{figure}
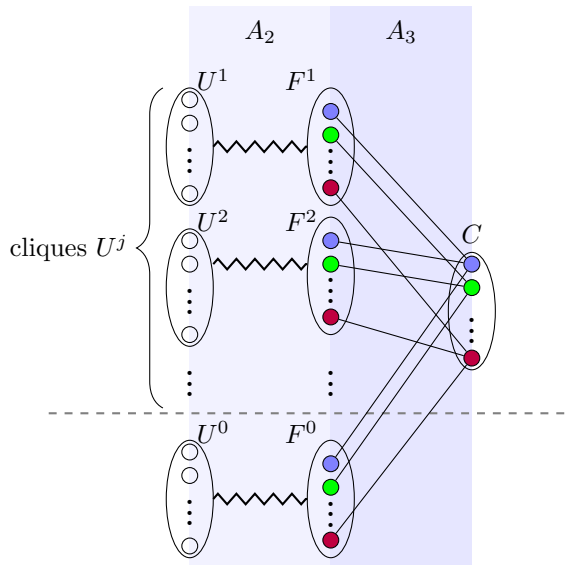

In this section, we derive arithmetical pruning rules (compare Theorem~\ref{firstprun}) from the flow model, describing necessary conditions for a partial coloring to be extendable. These can then be incorporated into an algorithm, without requiring to explicitly solve a flow problem. For this purpose, we focus on the network model for the extendability of the mixed approach without the source node $s$ and the target node $t$, compare Figure~\ref{fig:modflow}. We formulate the flow problem as LP and apply a theorem of Hoffman~\cite{hoffman1960} to derive said conditions. This way, we obtain generalized Hall conditions which give conditions for the existence of such flow.
\bigskip\\
In the course of this section, we show that three families of conditions are already necessary and sufficient. These are presented in Conclusion~\ref{fam:1}, in Conclusion~\ref{fam:2}, and in Conclusion~\ref{fam:3}. 
\bigskip\\
We require some additional notation: 
let $F^{C}=\bigcup_{i=1}^{l}F^i$ be the intermediate nodes corresponding to all uncolored vertices in the cliques $j=1,\ldots,l$. For two sets of nodes $W_1,W_2\subseteq V$, we write $e(W_1,W_2)$ as the \textit{number of arcs} going from $W_1$ to $W_2$. Finally, for any set of nodes $W\subseteq V$, let 
\begin{align*}
&\delta^+(W):=\{v\in V\mid wv\in A \text{ and } w \in W\} \text{ and} \\
&\delta^-(W):=\{v\in V\mid vw\in A \text{ and } w \in W\}
\end{align*}
denote the positive respectively the negative neighbors of $W$ and let $x\in\mathbb{R}^m.$ If
\begin {subequations}
\begin{align}
&1 \leq\sum_{\substack{v\in F\\uv\in A_2}}x_{uv}\leq 1 &&\forall\; u\in U \label{mip:flow:first}\\
&0 \leq x_{vi} -  \sum_{\substack{u\in U(\Pi_p)\\uv\in A_2}}x_{uv} \leq 0 &&\forall\; v\in F \text{ and } vi \in A_3 \\
&-\left\lceil \frac{n}{k_0} \right\rceil+|C_i| \leq - \sum_{\substack{v\in F\\vi\in A_4}}x_{vi}\leq-\left\lfloor\frac{n}{k_0}\right\rfloor+|C_i| &&\forall\;i=1,\dots,k_0
\end{align}
and
\begin{align}
& 0\leq x_{uv} \leq 1 &&\forall\;u\in U, v\in F \text{ with } uv\in F \\
& 0\leq x_{uv} \leq 1 && \forall\; u\in F^{C}, v\in C\\
& 0\leq x_{uv} \leq |U^0| && \forall\; u \in F^0, v\in C \label{mip:flow:last}
\end{align}
\end{subequations}
holds, then $x$ is a feasible flow vector of the simplified model, compare Figure~\ref{fig:modflow}. In the following, let $a\leq \mathcal{A}x \leq b$ with $c\leq x \leq d$ denote Model~\eqref{mip:flow:first}~--~\eqref{mip:flow:last} in matrix notation for appropriately chosen $a,\mathcal{A},b,c,d$. Note that since  Model~\eqref{mip:flow:first}~--~\eqref{mip:flow:last} is a max flow problem, the matrix $\mathcal{A}$ is totally unimodular. We recall the following result:

\begin{Theorem}[Hall Conditions, Hoffman~\cite{hoffman1960}] \label{theo:hoffman}
Let  $\mathcal{A}\in \{0,\pm 1\}^{n \times m}$ be a totally unimodular matrix. The system 
\begin{equation*}
 a\leq \mathcal{A}x \leq b,\qquad c\leq x \leq d
\end{equation*}
has a (integral) solution if and only if 
\begin{align}
&\sum_{i:w_i=-1}a_i + \sum_{i:v_i=1}c_i \leq \sum_{i:w_i=1}b_i + \sum_{i:v_i=-1}d_i
\end{align}
holds for all $v\in\{0,\pm 1\}^m, w\in \{0,\pm 1\}^n $ with $w\mathcal{A}=v$.
\end{Theorem}
 Since many Hall type theorems can be reduced to it, Hoffman referred to his result as \textit{``the most general theorem of the Hall type''}. In our case, it is $c\equiv 0$ and hence, the condition boils down to
\begin{align}
\sum_{i:w_i=-1}a_i \leq \sum_{i:w_i=1}b_i + \sum_{i:v_i=-1}d_i. \label{eq:hall} 
\end{align} 
Furthermore, note that:
\begin{Remark}\label{coroll:basic}
Let $w \in \{0,\pm 1\}^n$ be given. It is $w\mathcal{A}=v\in \{0,\pm 1\}^m$ if and only if $w_i=w_j$ or $w_i=0$ or $w_j=0$ for all $ij\in A$. 

A visual interpretation of the condition is as follows: We choose nodes from the network and assign to them a sign. These vertices yield the vector $w$ (non chosen vertices get the value zero). If we choose neighboring vertices, they must have the same sign. 
By $v=w\mathcal{A}$, we obtain a signed selection of edges (according to the coefficients $\pm1$ of $v$). Therefore, by selecting a node set with plus or minus one, $v$ corresponds to the incoming, respectively the outgoing arcs of this set. Both vectors $v,w$ together define one of the inequalities in~\eqref{eq:hall}.
\end{Remark}
In the following, we systematically investigate different choices of vertex sets and the corresponding generalized Hall conditions. Therefore, let $R_+, R_- \subseteq U, S_+,S_- \subseteq F,$ and $T_+,T_-\subseteq C$ such that 
\begin{align}
\delta^+(R_+)\cap S_-=\delta^+(R_-)\cap S_+=\delta^+(S_+)\cap T_-=\delta^+(S_-)\cap T_+=\emptyset.
\end{align}
Accordingly, define the vector
\begin{align}
\bar w:=1_{R_+ \cup S_+ \cup T_+  } - 1_{R_- \cup S_- \cup T_-  }, \label{vector:w}
\end{align}
 where $1_X$ denotes the incidence vector of the vertex set $X$. By construction, all possible choices of the vector $\bar w$ correspond exactly to the possible choices of the vectors $v$ in Theorem~\ref{theo:hoffman} and we obtain the following reformulation:  

\begin{Corollary}\label{theo:hall:2}
The flow problem~\eqref{mip:flow:first}~--~\eqref{mip:flow:last} has a solution if and only if for all $R_+, R_- \subseteq U, S_+,S_- \subseteq V_1,$ and $T_+,T_-\subseteq C$ with
\begin{align*}
\delta^+(R_+)\cap S_-=\delta^+(R_-)\cap S_+=\delta^+(S_+)\cap T_-=\delta^+(S_-)\cap T_+=\emptyset
\end{align*}
it holds that
\begin{align*}
|R_-|+\sum_{i \in T_-}\left(-\left\lceil \frac{n}{k_0} \right\rceil+|C_i| \right) \leq & |R_+|  +\sum_{i \in T_+}\left(-\left\lfloor \frac{n}{k_0} \right\rfloor+|C_i| \right)\\ 
& + e(R_-,F\setminus S_-) \\
& + e(U\setminus R_+,S_+) \\ 
& + e(S_-\cap F^{C},C\setminus T_-)\numberthis \label{eq:hall2} \\
& +|U^0| \cdot e(S_-\cap F^0,C\setminus T_-) \\
& + e(F^{C}\setminus S_+,T_+)\\
& +|U^0| \cdot e(F^0\setminus S_+,T_+).
\end{align*} 
\end{Corollary}

We formulate a basic observation:

\begin{Corollary}\label{remark:sum}
Given a $R_+, R_-, S_+,S_- T_+,T_-$ as in Corollary~\ref{theo:hall:2}, Condition~\eqref{eq:hall2} is dominated by the two conditions
\begin{align*}
&\begin{array}{lll}
0 &\leq & |R_+|  +\displaystyle\sum_{i \in T_+}\left(-\left\lfloor \frac{n}{k_0} \right\rfloor+|C_i| \right) + e(U\setminus R_+,S_+)  \\ 
& & + e(F^{C}\setminus S_+,T_+) +|U^0| \cdot e(F^0\setminus S_+,T_+)
\end{array}\numberthis 
\label{eq:hall+} \\
&\begin{array}{lll}
|R_-|+\displaystyle\sum_{i \in T_-}\left(-\left\lceil \frac{n}{k_0} \right\rceil+|C_i|\right) &\leq &e(R_-,F\setminus S_-) + e(S_-\cap F^{C},C\setminus T_-) \\ 
&& +|U^0| \cdot e(S_-\cap F^0,C\setminus T_-) 
\end{array}\numberthis \label{eq:hall-}
\end{align*} 
where Condition~\eqref{eq:hall+} corresponds to Condition~\eqref{eq:hall2} for $R_+$, $S_+$ and $T_+$ as given above and $R_-=S_- =T_- =\emptyset$ and Condition~\eqref{eq:hall-} corresponds to Condition~\eqref{eq:hall2} for $R_-$, $S_-$ and $T_-$ as given above and $R_+=S_+ =T_+ =\emptyset$. 

\end{Corollary}
Therefore, w.r.t Condition~\eqref{eq:hall2}, it is sufficient to focus on the two vectors, respectively on the inequalities, induced by
\begin{align*}
\bar w_+:=1_{R_+ \cup S_+ \cup T_+  }  \qquad\text{and}\qquad \bar w_-:=- 1_{R_- \cup S_- \cup T_-  }.
\end{align*}

We call a vector $w$, respectively a collection of signed vertices \textit{dominating} w.r.t. another vector $\tilde w$, if its induced inequality dominates the inequality induced by $\tilde w$. At first, we consider the conditions which are induced by $\bar w_+$.

\subsection{Selecting vertices with a positive contribution}

We consider inequalities which arise from vectors of the form $w_+:=1_{R_+ \cup S_+ \cup T_+  }$. Therefore, we set $R_-=S_-=T_-=\emptyset$ in the remainder of this subsection.

\begin{Lemma}\label{cor:hall+1}
Given vertex sets $R_+, S_+, T_+$ as defined above and let $$\bar R_+:=\{v\in U\mid \delta^+(v)\cap S_+\neq \emptyset\}.$$  Inequality~\eqref{eq:hall+} (induced by $w_+$) is dominated by the inequality induced by 
$1_{\bar R_+ \cup S_+ \cup T_+ }.$
\end{Lemma}
\begin{proof}
We divide the proof into two steps:
\begin{itemize}
	\item[1.] Assume that there is a vertex $v\in U$ with $v\notin R_+$ but $v\in \delta^-(S_+)$. Obviously, it is $|R_+|=|R_+\cup\{v\}|-1$ and
\begin{align*}
e(U\setminus (R_+\cup\{v\}),S_+)\leq e(U\setminus R_+,S_+)-1.
\end{align*}
Since all other terms are unchanged, $1_{(R_+\cup\{v\}) \cup S_+ \cup T_+  }$ dominates $1_{R_+ \cup S_+ \cup T_+  }$.

\item[2.]
Assume that there is a vertex $v\in U$ with $v\in R_+$ but $\delta^+(v)\cap S_+= \emptyset$. Since $|R_+|=|R_+\setminus\{v\}|+1$ and  
\begin{align*}
e(U\setminus R_+,S_+)= e(U\setminus (R_+\setminus\{v\}),S_+),
\end{align*}
the inequality $1_{(R_+\setminus\{v\}) \cup S_+ \cup T_+  }$ dominates $1_{R_+ \cup S_+ \cup T_+  }$ because all other terms remain unchanged.
\end{itemize}
\end{proof}
Restating the lemma, $R_+$ contains uncolored vertices which \textit{can} still be colored by a color in $S_+$. It implies that w.l.o.g., $R_+$ can be chosen maximally in that sense. In the light of this result, we assume w.l.o.g. that $R_+=\bar{R}_+$ in the following. Therefore, Inequality \eqref{eq:hall+} collapses to
\begin{align}\label{eq:hall+2}
\sum_{i \in T_+}\left(\left\lfloor \frac{n}{k_0} \right\rfloor-|C_i|\right)  \leq  |R_+|  + e(F^{C}\setminus S_+,T_+) +|U^0| \cdot e(F^0\setminus S_+,T_+) 
\end{align} 
In the case $T_+=\emptyset$, we obtain the inequality $0\leq|R_+|$ which is trivially fulfilled. Therefore, we assume that $T_+\neq \emptyset$ for the remainder of this section. The next result further classifies the dominating inequalities of the form \eqref{eq:hall+2}.

\begin{Lemma}\label{cor:hall+2}
Let $T_+\neq \emptyset$ be given. Assume that for any $S_+$, $R_+$ is chosen accordingly to Lemma~\ref{cor:hall+1}. A choice for $S_+$ such that 
\begin{align*}
1.)\quad S_+\subseteq\delta^{-}(T_+)\qquad\textnormal{and}\qquad 2.)\quad S_+\cap F^{0}\supseteq\delta^{-}(T_+)\cap F^{0} 
\end{align*}
holds dominates any other choice of $S_+.$
\end{Lemma}
\begin{proof}
We divide the proof into two steps:
\begin{itemize}
\item[1.] Suppose that there is a vertex $v\in S_+$ with $v\notin \delta^-(T_+)$. We remove $v$ from $S_+.$ After possibly changing the set $R_+$ (accordingly to Lemma~\ref{cor:hall+1}), we obtain that $1_{R_+\cup (S_+\setminus\{v\}) \cup T_+  } $ dominates $1_{R_+ \cup S_+ \cup T_+  }$, because all values remain unchanged except $|R_+|$ which possibly decreases. 
\item[2.] 
Assume that there is $v\in \delta^{-}(T_+)\cap F^{0}$ with  $v\notin S_+$. We add $v$ to $S_+$ and adapt $R_+$ accordingly. Therefore, $1_{R_+\cup (S_+\cup\{v\}) \cup T_+  }$ dominates $1_{R_+ \cup S_+ \cup T_+  }$, because $|R_+|$ increases at most by $|U^0|$ whereas $e(F^0\setminus S_+,T_+)$ decreases by $|U^0|$ since $v$ has exactly one outgoing arc.
\end{itemize}
\end{proof}
The lemma implies that w.l.o.g., $S_+$ contains \textit{at most} the vertices connected to $T_+$. Furthermore, concerning the vertices in $F^0$, it contains \textit{exactly} all neighbors of $T_+$. Hence, w.l.o.g., $e(F^0\setminus S_+,T_+)=0$ and we obtain:
\begin{Conclusion}\label{fam:1}
With respect to Corollary~\ref{theo:hall:2} and Remark~\ref{remark:sum}, it suffices to consider  $T_+\neq \emptyset$ and $S_+$ such that $S_+\subseteq \delta^-(T_+)$ and $S_+\cap F^0=\delta^{-}(T_+)\cap F^{0}$. Hence, the condition boils down to:
\begin{align}
\sum_{i \in T_+}\left(\left\lfloor \frac{n}{k_0} \right\rfloor-|C_i|\right)  \leq  |R_+|  + e(F^{C}\setminus S_+,T_+),  \label{eq:hall+3}
\end{align} 
where $R_+$ is chosen depending on $S_+$ as described in Lemma~\ref{cor:hall+1}.
\end{Conclusion}
We point out that an inequality induced by $T_+:=T_+^1 \cup T_+^2$ is, in general, not dominated by the corresponding inequalities induced by $T_+^1$ and $T_+^2$ since $|R_+|$ w.r.t. $T_+$ is, in general, smaller than $|R_+^1|+|R_+^1|$ (w.r.t. $T_+^1$ and $T_+^2$).

Concerning an application as pruning rule, it is clear that any choice of $T_+$ and $S_+$ yields a condition which has to be fulfilled for a coloring to be extendable. 

For a practical point of view, consider that $T_+$ corresponds to a selection of colors and that $S_+=\delta^-(T_+)$. In this case, there are no edges going from $F^C\setminus S_+$ to $T_+$ and $R_+$ contains all vertices in $U$ which \textit{can} be colored by colors associated to $T_+$, i.e., with a slight abuse of notation, the condition writes:
\begin{align}
\sum_{i \in T_+}\left(\left\lfloor \frac{n}{k_0} \right\rfloor-|C_i|\right)  \leq  \left|\{v\in U\mid F(v)\cap T_+\neq \emptyset\}\right|.\label{com:cond+}
\end{align}  
Or, from a combinatorial point of view: There have to be enough vertices which can be colored by colors of $T_+$ such that the corresponding color classes can be filled up to their cardinality requirement. We remark that this inequality can (theoretically) be strengthened. 
Clearly, such strengthening is difficult to realize in practice, since it relies on subset based arguments. 

However, with respect to our computational experiments, we discuss this strengthening for the case $|T_+|=1$. Hence, assume that $T_+=\{f\}$ for 'some color $f$'. Then, Condition~\eqref{com:cond+} reads   
\begin{align}
\left\lfloor \frac{n}{k_0} \right\rfloor-|C_f|  \leq  
&\left|\{v\in U\mid f\in F(v)\}\right| \notag\\
&=\sum_{j=1}^{l}\left|\{v\in U^j\mid f\in F(v)\}\right|+\left|\{v\in U^0\mid f\in F(v)\}\right|.
\end{align}  
Note that, in the current setting $S_+\cap F^j$ contains a single element. The decision whether the inequality can be strengthened by excluding this element from $S_+$ 
can be encoded in the condition. We obtain:
\begin{align}
\left\lfloor \frac{n}{k_0} \right\rfloor-|C_f|  \leq  &\sum_{j=1}^{l}\min\left\{1,\left|\{v\in U^j\mid f\in F(v)\}\right|\right\} \notag\\
+&\left|\{v\in U^0\mid f\in F(v)\}\right|. \label{f1}
\end{align}
That is, the selected color classes have to be filled up with at most one element per clique and all vertices in $U^0$ which can use these colors. 

\subsection{Selecting vertices with a negative contribution}

We consider inequalities which arise from vectors of the form $w_-:=-1_{R_- \cup S_- \cup T_-  }$. Therefore, we set $R_+=S_+=T_+=\emptyset$ in the remainder of this subsection. As in the previous subsection, we show that $R_-$ can be chosen accordingly to $S_-$.

\begin{Lemma}\label{cor:hall-1}
Given vertex sets $R_-, S_-, T_-$ as defined above and let $$\bar R_-:=\{v\in U\mid \delta^+(v)\subseteq S_-\}.$$ Inequality \eqref{eq:hall-} induced by $w_-$ is dominated by the one induced by 
$1_{\bar R_- \cup S_- \cup T_- }.$  
\end{Lemma}
\begin{proof}
As above, we divide the proof into two parts:
\begin{itemize}	
\item[1.] If there is $v\in U$ with $v\notin R_-$ and $\delta^+(v)\subseteq S_-$, the inequality induced by $- 1_{(R_- \cup\{v\})\cup S_- \cup T_-  }$ dominates the inequality corresponding to $- 1_{R_- \cup S_- \cup T_-  }$ because $|R_- \cup\{v\}|>|R_-|$ and $e(R_-\cup\{v\},V_1\setminus S_-)=e(R_-,V_1\setminus S_-)$ hold. The remaining terms stay unchanged.

\item[2.]Similarly, assume that there is $v\in U$ with $v\in R_-$ and there is $u\in\delta^+(v)$ with $u\notin S_-$. Then, the inequality induced by $- 1_{(R_- \setminus\{v\})\cup S_- \cup T_-  }$ dominates the one corresponding to $- 1_{R_- \cup S_- \cup T_-  }$ because $|R_- \setminus\{v\}|=|R_-|-1$ and $e(R_-\setminus\{v\},V_1\setminus S_-)\leq e(R_-,V_1\setminus S_-)-1.$ Again, the remaining  terms stay unchanged.
\end{itemize}
\end{proof}
We rephrase the result. $\bar R_-$ contains \textit{all} vertices in $U$, for which \textit{all} outgoing arcs are in $S_-$. This is the set of nodes which \textit{have} to be colored by a color in $S_+$. Lemma~\ref{cor:hall-1} implies that w.l.o.g., $R_-$ can be chosen maximally in that sense. Hence, we obtain $e(R_-,V_1\setminus S_-)=0$ and the inequality collapses to
\begin{align*}
|R_-|+\sum_{i \in T_-}\left(-\left\lceil \frac{n}{k_0} \right\rceil+|C_i|\right)  \leq &  e(S_-\cap F^{C},C\setminus T_-) \\ 
& +|U^0| \cdot e(S_-\cap F^0,C\setminus T_-) \numberthis \label{eq:hall-2}
\end{align*} 

In the next paragraphs we discuss the two cases $T_-=\emptyset$ and $T_-\neq \emptyset$ with the choice of $R_-$ according to Lemma~\ref{cor:hall-1}.

\paragraph{The first case.}

We start with $T_-=\emptyset$, for which Inequality~\eqref{eq:hall-2} reduces to 
\begin{align}
|R_-|& \;\leq\;e(S_-\cap F^{C},C) + |U^0| \cdot e(S_-\cap F^0,C). \label{eq:hall-3}
\end{align} 

This inequality is dominated by inequalities for each set $U^j$ :
\begin{Corollary}
Inequality \eqref{eq:hall-3} induced by $- 1_{R_- \cup S_- \cup \emptyset }$ is dominated by the (atomic) inequalities induced by
\begin{align*}
- 1_{R_-\cap U^j \cup S_-\cap F^j \cup \emptyset  }    \qquad\forall j=1,\ldots,l.     
\end{align*}
\end{Corollary}
\begin{proof}
Observe that the sum of the atomic inequalities is exactly 
\begin{align} \label{eq:hall-4}
|R_-|& \leq e(S_-\cap F^{C},C).
\end{align} 
If $S_-\cap F^0\neq \emptyset,$ the contribution to $|R_-|$ on the left side was at most $|U^0|$. However, the contribution of $|U^0| \cdot e(S_-\cap F^0,C)$ on the right side is at least $|U_0|$. Hence, Condition~\eqref{eq:hall-3} is implied by Condition~\eqref{eq:hall-4}.
\end{proof}

Therefore, the dominating conditions solely depend on the sets $F^j$. Note that from any element in $F^j$ exactly one arc goes to $C$. All in all, for $T_-=\emptyset$, the strongest conditions are given by: 

\begin{Conclusion}\label{fam:2}
With respect to Corollary~\ref{theo:hall:2} and Remark~\ref{remark:sum}, for $T_-=\emptyset$, it suffices to consider $S_-\subseteq F^j$ for all $j=1,\ldots,l$, that is, the conditions
\begin{align}
|R_-|& \leq |S_-|\label{eq:hall--}
\end{align} 
where $R_-$ is chosen depending on $S_-$ as described in Lemma~\ref{cor:hall-1}.
\end{Conclusion}


We remark that Condition~\eqref{eq:hall--} corresponds to the well known Hall Conditions for matchings in bipartite graphs: $S_-$ corresponds to selection of colors, and $R_-$ contains exactly these nodes which have to be colored by these colors. Clearly there have to be more colors than vertices to permit a coloring. 

%
%
%

\paragraph{The second case}

We consider the case $T_-\neq\emptyset$ and discuss Inequality~\eqref{eq:hall-2}. Analogously to the 'positive' case, we state a first Corollary:

\begin{Lemma}\label{cor:hall-2}
Let $T_-\neq \emptyset$ be given. Assume that for any $S_-$, $R_-$ is chosen according to Lemma~\ref{cor:hall-1}. A choice for $S_-$ such that 
\begin{align*}
1.)\quad S_-\supseteq\delta^{-}(T_-)\qquad \textnormal{and}\qquad 2.)\quad S_-\cap F^{0}\subseteq\delta^{-}(T_-)\cap F^{0}
\end{align*}
holds dominates any other choice of $S_-$.
\end{Lemma}
\begin{proof}
As before, we divide the proof into two steps:
\begin{itemize}
\item[1.] Suppose that there is a vertex $v\in \delta^{-}(T_-)$ with $v\notin S_-$. We add $v$ to $S_-$. After possibly changing the set $R_-$ (accordingly to Lemma \ref{cor:hall-1}), $-1_{R_-\cup (S_-\cup\{v\}) \cup T_-  }$ dominates $-1_{R_- \cup S_- \cup T_-  }$, because all values remain unchanged except $|R_-|$ which possibly increases. 

\item[2.] 
Now, assume that there is $v\in S_-\cap F^0$ with  $v\notin \delta^{-}(T_-)\cap F^{0}$. We remove $v$ from $S_-$ and adapt $R_-$ accordingly. Therefore, $-1_{R_-\cup (S_-\setminus\{v\}) \cup T_-  }$ dominates $-1_{R_- \cup S_- \cup T_-  }$, because $|R_-|$ (left hand side) decreases at most by $|U^0|$ whereas the right side ($|U^0| \cdot e(S_-\cap F^0,C\setminus T_-)$) also decreases by $|U^0|$ since $v$ has exactly one outgoing arc to $C\setminus T_-$. 
\end{itemize}
\end{proof}
The lemma implies that w.l.o.g., $S_-$ contains \textit{at least} the vertices connected to $T_-$, and w.r.t. $U^0$, contains exactly the vertices connected to $T_-$. Hence, w.l.o.g., $e(S_-\cap F^0,C\setminus T_-)=0$. We obtain:
\begin{Conclusion}\label{fam:3}
With respect to Corollary~\ref{theo:hall:2} and Remark~\ref{remark:sum}, it suffices to consider  $T_-\neq \emptyset$ and $S_-$ such that $S_-\supseteq \delta^-(T_-)$ and $S_-\cap F^0=\delta^{-}(T_-)\cap F^{0}$. Hence, the condition boils down to
\begin{align}
|R_-|+\sum_{i \in T_-}\left(-\left\lceil \frac{n}{k_0} \right\rceil+|C_i|\right)  \leq \; e(S_-\cap F^{C},C\setminus T_-)
\end{align} 
where $R_-$ is chosen depending on $S_-$ as described in Lemma~\ref{cor:hall-1}.
\end{Conclusion}
We point out that an inequality induced by $T_-:=T_-^1 \cup T_-^2$ is, in general, not dominated by the corresponding inequalities induced by $T_-^1$ and $T_-^2$. 
%
Again, we cannot state any further, general dominance relations. 
Still, ``combinatorial'' rules can be obtained straightforward, i.e., 
For a practical application, consider that $T_-$ corresponds to a selection of colors and that $S_-=\delta^-(T_-)$. In this case, there are no edges going from $S_-\cap F^C$ to $C\setminus T_-$ and $R_-$ contains all vertices in $U$ which \textit{have} to be colored by colors associated to $T_-$, i.e., with a slight abuse of notation, the condition writes:
\begin{align}
\left|\left\{v\in U\mid F(v)\subseteq T_-\right\}\right|\leq \sum_{i \in T_-}\left(\left\lceil \frac{n}{k_0} \right\rceil-|C_i|\right) . \label{f3}
\end{align}  
Or, from a combinatorial point of view: The number of vertices which have to be colored by colors in $T_-$ may not exceed the maximal cardinality of the color classes associated to $T_-$. In contrast to Condition~\eqref{f1}, no further strengthening via minimization subproblems is necessary due to the presence of the Hall Conditions~\eqref{eq:hall--}.

\section{Computational Study}\label{sec:compu}

\subsection{Setting} \label{diapp}

In this section, we provide a computational evaluation of our results. As a basic solution algorithm, we consider EQDSATUR as presented in~\cite{MNS:2014b} (henceforth: STD). Note that this includes a check of the pruning rule~\eqref{easyprun} as described in Theorem~\ref{firstprun} at step~\eqref{choosei} in EQDSATUR. We compare the key figures/statistics of this algorithm to the same algorithm extended with the pruning rules described in this work. 
\bigskip\\
In particular, we consider the algorithm FLOW. FLOW extends on EQDSATUR by, in step~\eqref{choosei}, after checking the rule~\eqref{easyprun}, executing the Pruning Algorithm~\ref{alg:prune}. Note that this takes place in every node of the search tree and that in step~\eqref{chooseDecomp}, the pruning algorithm calls Algorithm~\ref{alg:nonadjcliques} to determine a decomposition of the uncolored vertices.


Similar, we consider the algorithm COMB. COMB works exactly like FLOW, however, instead of solving a flow problem in step~\eqref{solveflow} of Algorithm~\ref{alg:prune}, we evaluate the pruning rules~\eqref{f1}, \eqref{eq:hall--}, and \eqref{f3}. Thereby, to achieve a fast runtime, we (only) consider sets $T_-$ and $T_+$ including a \textit{single} or \textit{all but one} color. 


This way, FLOW contains the strongest conditions but requires the solution of flow problems at every node of the search tree whereas COMB contains only a subset of conditions which can be tested arithmetically.

For all our tests, we assume that upper ($\bar{k}$) and lower bounds ($\underline{k}$) on $\chi_{eq}(G)$ are precomputed. In cases where $\underline{k}=\bar k$, we say that both algorithms solve the instance in $0$ seconds. Both algorithms are initialized with a partial coloring of $G$, obtained by (greedily) coloring a maximal clique.

For both, the FLOW and the COMB algorithm, the extended pruning rules are inserted in Step~\ref{choosei} in the description of EQDSATUR, i.e., at first the pruning rule given by Theorem~\ref{firstprun} is evaluated and if no pruning occurs, the pruning rules derived from the {\it mixed approach} are evaluated. By this setup, every instance solved by FLOW resp. COMB will require at most as many nodes in the search tree as if solved by STD. 

We employ two groups of test instances. The first one is made up from random instances, for the second group we consider instances from the DIMACS challenge. The random instances are generated according to the Erd\H{o}s-R\'enyi graph model $G(n,p)$ model for 
\begin{align*}
n\in\left\{40+5i\mid i=0,\ldots,6\right\} \text{ and } p\in\left\{0.1\cdot i\mid i=1,\ldots,9\right\}.
\end{align*}
For each combination of $n$ and $p$, we generate a class of $200$ test instances. We do not regard smaller instances as the running times are mostly insignificant below $n=40$. For the case $n>70$ the numbers of the instances which exceed the time limit becomes too large, such that these are neglected as well.
\bigskip\\
For all our computations, we employ an Intel(R) Core(TM) i7-3770 CPU @ 3.40 GHz with 32 GB RAM and a time limit of $3,600$ seconds. For all algorithms, we developed our own C++ code. The 
flow problems are solved by the Push Relabel algorithm by Goldberg (cf.~\cite{goldberg1985new}) as provided in the Boost Graph Library for C++~\cite{boost_graph}. 
Note that flow problems with lower bounds on the flow values can be solved by {\it two} max flow problems (see e.g. 
\cite{AMO:1993}). The code for STD was used as a basis for the code for the other algorithms, providing comparability among both algorithms.
\bigskip\\
The remainder of this section is structured as follows. At first, in Subsection~\ref{results1}, we discuss the results of the random instances with $n\leq 65$. Due to the size of this testbed ($10,800$ instances), we consider this as our most significant study. In Subsection~\ref{results4}, we comment on results of larger instances ($n=70$). In the following, in Subsection~\ref{results2}, we focus on specific instances from this test set (with $n=60$) and evaluate more economic interpretations of our pruning rule, i.e, evaluating the pruning rules every $x^{th}$ nodes, etc. Finally, in Subsection~\ref{results3}, we focus on the results of the DIMACS instances.

\subsection{Random Graphs} \label{results1}
\begin{table}
\caption{Solution time, number of instances meeting the time limit, and number of nodes in the  tree of STD, FLOW and COMB for the random instances.}
\label{tab:all}
\setlength{\tabcolsep}{2pt}
\renewcommand{\arraystretch}{0.75}
\begin{center}
{\footnotesize 
\begin{tabular}{l l r r r r r r r r r}
\toprule
\multirow{2}{*}{\textbf{}} & \multirow{2}{*}{\textbf{p}} & \multicolumn{3}{c}{\textbf{Time (s)}} & \multicolumn{3}{c}{\textbf{\# Timeout}} & \multicolumn{3}{c}{\textbf{\# Nodes}} \\
\cmidrule(lr){3 - 5}\cmidrule(lr){6-8}\cmidrule(lr){9-11}
& & STD & FLOW & COMB & STD & FLOW & COMB & STD & FLOW & COMB \\
\midrule
\multirow{9}{*}{\rotatebox[origin=c]{90}{n=40}}&0.1&0.0&0.0&0.0&0&0&0&455.6&80.8&86.7\\
 &0.2&0.0&0.0&0.0&0&0&0&590.6&156.1&191.3\\
 &0.3&0.0&0.0&0.0&0&0&0&5269.2&448.9&565.2\\
 &0.4&0.0&0.0&0.0&0&0&0&37373.8&509.3&627.6\\
 &0.5&2.7&0.1&0.0&0&0&0&2773741.9&1241.4&1487.0\\
 &0.6&2.8&0.1&0.0&0&0&0&3008095.2&1476.2&1813.8\\
 &0.7&0.2&0.0&0.0&0&0&0&185451.2&1030.2&2747.0\\
 &0.8&0.0&0.0&0.0&0&0&0&709.8&281.3&523.0\\
 &0.9&0.0&0.0&0.0&0&0&0&5928.1&40.8&249.5\\
\multicolumn{2}{c}{\textbf{Avg.}} &\textbf{0.6}&\textbf{0.0}&\textbf{0.0}&\textbf{0.0}&\textbf{0.0}&\textbf{0.0}&\textbf{668,623.9}&\textbf{585.0}&\textbf{921.2}\\
\midrule
\multirow{9}{*}{\rotatebox[origin=c]{90}{n=45}}&0.1&0.0&0.0&0.0&0&0&0&1908.2&263.1&308.3\\
 &0.2&0.0&0.0&0.0&0&0&0&13198.2&174.6&215.2\\
 &0.3&0.0&0.0&0.0&0&0&0&3402.2&535.6&589.7\\
 &0.4&21.9&18.2&4.9&0&1&0&8866647.0&6781.2&8782.8\\
 &0.5&11.7&0.2&0.0&0&0&0&10992516.2&4268.7&4909.9\\
 &0.6&18.3&0.2&0.0&1&0&0&284477.5&2896.6&3611.6\\
 &0.7&28.4&0.1&0.0&0&0&0&24350071.8&2019.2&5033.9\\
 &0.8&0.0&0.1&0.0&0&0&0&21546.9&1003.8&2150.5\\
 &0.9&0.5&0.0&0.1&0&0&0&424776.5&92.7&19766.0\\
\multicolumn{2}{c}{\textbf{Avg.}} &\textbf{9.0}&\textbf{2.1}&\textbf{0.6}&\textbf{0.1}&\textbf{0.1}&\textbf{0.0}&\textbf{4,995,393.8}&\textbf{2,003.9}&\textbf{5,040.9}\\
\midrule
\multirow{9}{*}{\rotatebox[origin=c]{90}{n=50}}&0.1&0.0&0.0&0.0&0&0&0&2030.0&234.7&243.9\\
 &0.2&0.0&0.0&0.0&0&0&0&2840.7&567.0&716.4\\
 &0.3&1.1&0.3&0.0&0&0&0&1351642.0&6750.5&8113.2\\
 &0.4&0.3&0.4&0.1&0&0&0&313280.7&5887.0&6453.0\\
 &0.5&8.8&2.0&0.3&0&0&0&7991515.2&35808.6&41327.3\\
 &0.6&70.4&0.9&0.2&3&0&0&16254662.0&13345.3&15825.2\\
 &0.7&125.8&0.8&0.6&5&0&0&31348814.6&11962.5&98823.1\\
 &0.8&0.1&0.3&0.3&0&0&0&62195.5&3445.3&19294.1\\
 &0.9&4.1&0.0&0.8&0&0&0&3177683.6&459.4&212998.6\\
\multicolumn{2}{c}{\textbf{Avg.}} &\textbf{23.4}&\textbf{0.5}&\textbf{0.3}&\textbf{0.9}&\textbf{0.0}&\textbf{0.0}&\textbf{6,722,740.5}&\textbf{8,717.8}&\textbf{44,866.1}\\
\midrule
\multirow{9}{*}{\rotatebox[origin=c]{90}{n=55}}&0.1&0.0&0.0&0.0&0&0&0&1420.5&170.9&214.6\\
 &0.2&18.0&18.1&4.5&1&1&0&4986.1&1700.1&2005.0\\
 &0.3&0.0&0.3&0.0&0&0&0&6158.1&4165.0&5136.4\\
 &0.4&1.6&6.3&0.9&0&0&0&1347000.0&110532.8&121321.7\\
 &0.5&22.0&24.5&22.3&1&1&1&3163743.9&107644.0&1468752.4\\
 &0.6&6.8&3.6&7.4&0&0&0&5027920.1&47315.7&1452891.4\\
 &0.7&355.6&21.5&54.6&18&1&2&27969583.3&45276.3&5517768.0\\
 &0.8&0.5&0.9&0.9&0&0&0&260490.8&10760.7&51901.5\\
 &0.9&150.2&0.1&4.3&4&0&0&59065414.0&924.2&953949.7\\
\multicolumn{2}{c}{\textbf{Avg.}} &\textbf{61.6}&\textbf{8.4}&\textbf{10.5}&\textbf{2.7}&\textbf{0.3}&\textbf{0.3}&\textbf{10,760,746.3}&\textbf{36,498.9}&\textbf{1,063,771.2}\\
\midrule
\multirow{9}{*}{\rotatebox[origin=c]{90}{n=60}}&0.1&0.0&0.0&0.0&0&0&0&4939.3&660.4&816.6\\
 &0.2&0.0&0.3&0.0&0&0&0&32930.1&5757.9&7477.2\\
 &0.3&0.0&1.2&0.2&0&0&0&21821.2&14859.4&15582.7\\
 &0.4&20.3&25.3&19.2&1&1&1&1888033.8&142547.2&203785.6\\
 &0.5&93.6&50.7&15.1&4&1&0&14972786.2&460700.9&537577.0\\
 &0.6&264.3&25.1&7.0&10&0&0&61405164.3&358795.3&927542.2\\
 &0.7&286.1&16.6&22.8&14&0&1&26583895.8&198475.6&264492.5\\
 &0.8&69.9&21.2&21.2&2&1&1&19784313.1&36983.3&167645.9\\
 &0.9&381.2&0.6&18.4&20&0&1&15728760.9&8988.0&19887.0\\
\multicolumn{2}{c}{\textbf{Avg.}} &\textbf{123.9}&\textbf{15.7}&\textbf{11.5}&\textbf{5.7}&\textbf{0.3}&\textbf{0.4}&\textbf{15,602,516.1}&\textbf{136,418.7}&\textbf{238,311.9}\\
\midrule
\multirow{9}{*}{\rotatebox[origin=c]{90}{n=65}}&0.1&0.0&0.0&0.0&0&0&0&10189.7&726.7&907.8\\
 &0.2&0.0&0.3&0.1&0&0&0&23489.7&5005.2&6754.3\\
 &0.3&8.9&7.7&1.0&0&0&0&9065186.1&93530.0&105290.0\\
 &0.4&4.2&61.8&8.9&0&0&0&2756575.9&721947.8&768299.8\\
 &0.5&130.5&266.3&69.1&5&2&1&26782892.2&2697309.7&3127829.9\\
 &0.6&208.1&194.5&38.6&10&0&0&17936777.2&1981494.8&2216573.8\\
 &0.7&243.6&95.3&31.9&12&0&0&17971417.9&993957.4&1774304.9\\
 &0.8&139.7&17.9&127.3&3&0&3&42983183.5&181851.7&5421474.2\\
 &0.9&256.2&1.7&1.3&14&0&0&3008433.6&18286.5&37857.0\\
\multicolumn{2}{c}{\textbf{Avg.}} &\textbf{110.1}&\textbf{71.7}&\textbf{30.9}&\textbf{4.9}&\textbf{0.2}&\textbf{0.4}&\textbf{13,393,127.3}&\textbf{743,790.0}&\textbf{1,495,476.9}\\
\midrule
\multicolumn{2}{c}{\textbf{Avg.}} &\textbf{54.8}&\textbf{16.4}&\textbf{9.0}&\textbf{2.4}&\textbf{0.2}&\textbf{0.2}&\textbf{8,690,524.7}&\textbf{154,669.0}&\textbf{474,731.3}\\
\bottomrule
\end{tabular}
}
\end{center}
\end{table}

For the detailed results, we refer to Table~\ref{tab:all}. In this table, we present three key figures to evaluate the success of the additional pruning rules, i.e., of the algorithms FLOW and COMB.
\bigskip\\
At first, for each algorithm, we report the average runtime (Time) of each class of instances. Hereby, each instance which cannot be solved within the time limit accounts for $3,600$ seconds of computation time. This way, the total average time multiplied with $10,800$ yields the overall time required to deal with all instances. Obviously, the lower this value, the better. Additionally, we state the number of instances which \textit{cannot} be solved within the time limit per algorithm. At last, we report the average number of nodes in the search tree required per algorithm. Hereby, the nodes of an ``unsolved'' instance are left out in the averages of \textit{all} algorithms to allow a comparison of the required ``nodes to optimality'' of each algorithm. 
\bigskip\\
We observe that the additional pruning rules are very strong, having a dramatic effect on all three key figures of the DSATUR algorithms. Over all classes/instances, even a straightforward implementation of the exact pruning scheme (FLOW) can reduce the runtime from $54.8$ seconds down to $16.4$ seconds (STD) on average. This decrease is enabled by a drastic decrease in the node count since STD requires, on average, $56$ \textit{times} more nodes than FLOW (to solve an instance to optimality). However, it is clear that the time required per node in the search tree of FLOW is larger than the one of STD, as the former has to solve multiple flow problems in each of them and has to sustain additional data structures for evaluating these problems.

In comparison to FLOW, COMB is more lightweight, which results in a further decrease in solution time, down to $9.0$ seconds per instance on average. Since COMB only evaluates a part of the criterion's of the flow problem, it requires more nodes than FLOW, but this trade-off appears to be beneficial for the overall solution time.

A further improvement of FLOW, respectively COMB over STD is the greatly increased stability of the algorithm. While STD cannot solve, on average $2.4$ instances (out of $200$) within the time limit, both extended algorithms can solve all but $0.2$ of the instances. This is especially remarkable for the instance class $n=60$ and $p=0.9$, for which STD cannot solve $20$ instances whereas FLOW (COMB) can solve all (but one of) the instances. 
\bigskip\\
While, in general, the extended algorithms can greatly improve on STD, we remark that this trend does not hold for every instance class. For example, for $n=65$ and $p=0.4$, both extended algorithms require more time ($61.8$, respectively $8.9$ seconds) compared to STD ($4.2$ seconds). Nevertheless, the overall improvements are substantial, indicating the strength and the importance of the pruning rules described in this work.

\subsection{Large Instances} \label{results4}

In this subsection, we give some concluding comments on how our algorithms behave for larger instances. In particular, we consider random instances with $n=70$. For detailed results, we refer to Table~\ref{tab:seventy}. What we can see, in comparison to the smaller instances, compare Table~\ref{tab:all}, is that FLOW is not faster than STD anymore, whereas COMB is still the best algorithm. This is due to the problem size, which requires larger flow problems to be initialized and to be solved which is relatively more time consuming. However, COMB still outperforms the other two algorithms, both with respect to time and with respect to the number of solved instances. We refer to Figure~\ref{fig:risingn} for a visualization for the average solution time for increasing problem size, pointing out that COMB is always the fastest algorithm.
\begin{figure}
\pgfplotstableread[row sep=\\,col sep=&]{
nn 	&	STD	&	FLOW	&	COMB  \\
40	&	0.6	&	0.0	&	0.0 \\
45	&	9.0	&	2.1	&	0.6 \\
50	&	23.4	& 	0.5	& 	0.3 \\
55	&	61.6	&	8.4	&	10.5 \\
60	& 	123.9	&	15.7	&	5.7 \\
65	&	110.1	&	71.7	& 	30.9 \\
70	&	240.5	&	320.3	& 106.4 \\ 
}\mydata

\begin{tikzpicture}[scale=0.95]

    \begin{axis}[
            width=\textwidth,
            height=.5\textwidth,
            legend pos = north west,
            legend cell align=left,
	    every axis x label/.style={at={(current axis.south)},below=0.75cm},
            xtick=data,
            xlabel={Node size of the random instances.},
            ymin=0,ymax=350,
            ymajorgrids = true,
            ylabel={Average Solution Time (s)},
        ]
	\addplot[ thick,mark options={fill=blue!50}, mark=square*,mark size=3pt] table[x=nn,y=STD]{\mydata};
	\addplot[ thick,mark options={fill=red!50}, mark=square*,mark size=3pt] table[x=nn,y=FLOW]{\mydata};
	\addplot[ thick,mark options={fill=green}, mark=square*,mark size=3pt] table[x=nn,y=COMB]{\mydata};
        \legend{STD,FLOW,COMB}
    \end{axis}
\end{tikzpicture}
\caption{Average solution time of STD, FLOW and COMB for increasing problem size $n\in\{40,\ldots,70\}$.}
\label{fig:risingn}
\end{figure}
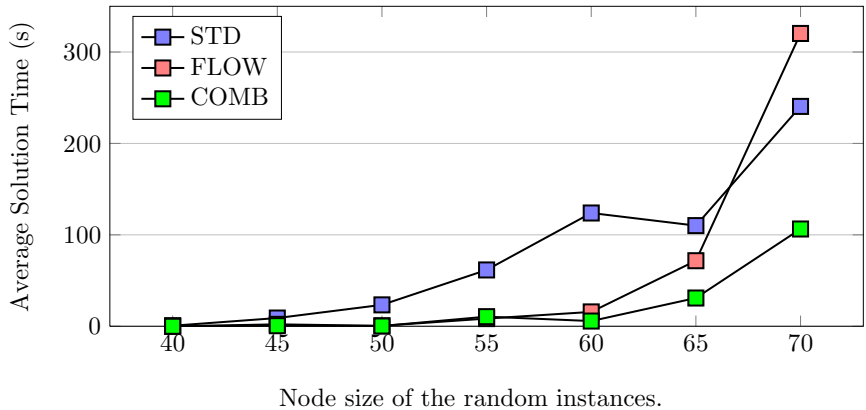
\bigskip\\
In the case that even larger instances are to be considered, also compare the results on the DIMACS instances in Subsection~\ref{results3}, the algorithms have to be adapted as the additional overhead due to the extended pruning rules rises. One way of doing so is by updating the pruning schemes lazily, as described in the next Subsection. 

In general, when the problem size increases, the necessary data structures and the time spent updating and searching in these increases as well. This way, custom implementations, especially tailored for larger networks become necessary. In this context, COMB could also be adapted to be even more lightweight by not considering the Pruning Rule~\eqref{f3} as this rule requires substantially more effort, from a data lookup perspective, than the other rules. All in all, the larger the problems become, the less likely it is that a single variant/implementation of an algorithm is equally well fit for both, small and very large problem sizes. Naturally, the same arguments apply to other parameters such as graph density as well, the more specific a certain testbed is, the more worthwhile is a custom implementation of an algorithm.

\begin{table}
\caption{Statistics for the random instances with $n=70$.}
\label{tab:seventy}
\setlength{\tabcolsep}{2pt}
\renewcommand{\arraystretch}{0.75}
\begin{center}
{\footnotesize 
\begin{tabular}{l l r r r r r r r r r}
\toprule
\multirow{2}{*}{\textbf{}} & \multirow{2}{*}{\textbf{p}} & \multicolumn{3}{c}{\textbf{Time (s)}} & \multicolumn{3}{c}{\textbf{\# Timeout}} & \multicolumn{3}{c}{\textbf{\# Nodes}} \\
\cmidrule(lr){3 - 5}\cmidrule(lr){6-8}\cmidrule(lr){9-11}
& & STD & FLOW & COMB & STD & FLOW & COMB & STD & FLOW & COMB \\
\midrule
 &0.1&0.1&0.1&0.0&0.0&0.0&0.0&106,795.1&4,809.9&6,041.6\\
 &0.2&0.0&0.4&0.1&0.0&0.0&0.0&31,006.6&5,109.0&7,570.1\\
 &0.3&0.5&21.4&2.8&0.0&0.0&0.0&320,911.9&222,538.0&237,933.1\\
 &0.4&72.6&283.8&48.1&2.0&1.0&0.0&23,949,135.2&2,716,385.9&2,822,804.6\\
 &0.5&390.9&533.1&129.9&17.0&5.0&2.0&59,033,446.8&4,557,973.7&4,894,282.7\\
 &0.6&443.5&1,468.1&422.0&16.0&31.0&2.0&44,720,718.5&9,962,971.2&11,377,565.5\\
 &0.7&275.0&472.9&217.0&11.0&2.0&0.0&37,178,971.1&3,823,696.4&6,857,377.4\\
 &0.8&872.9&93.8&112.9&37.0&2.0&2.0&127,179,557.7&520,247.9&2,337,771.0\\
 &0.9&108.6&8.8&24.7&6.0&0.0&1.0&295,475.8&83,459.3&165,311.7\\
\multicolumn{2}{c}{\textbf{avg.}}&\textbf{240.5}&\textbf{320.3}&\textbf{106.4}&\textbf{9.9}&\textbf{4.6}&\textbf{0.8}&\textbf{32,535,113.2}&\textbf{2,433,021.3}&\textbf{3,189,628.6}\\
\bottomrule
\end{tabular}
}
\end{center}
\end{table}

\subsection{Lazy Application} \label{results2}

In the previous subsection, we have observed that a DSATUR based algorithm for the ECP greatly benefits from additional pruning rules. However, such pruning rules induce additional effort on the computational side, i.e., through additional data structures which have to be updated or through the flow problems which have to be solved. In this context, a natural question is whether the runtime of our extended algorithms can benefit of a lazy application of such pruning schemes. We elaborate on this in the following.
\bigskip\\
In general, the additional time investment of the pruning rules depends on two factors. The first one is the repeated search for an improving clique decomposition (CD, compare Algorithm~\ref{alg:nonadjcliques}), the second one is creation and solution of the flow problems, respectively the evaluation of the combinatorial pruning rules. In the setting presented above, both actions are performed once per node of the  tree. Therefore, reasonable alternatives are to execute each of the actions only at every $x^{th}$ ($x>1$) node in the search tree. 

However, computational experiments show that evaluating the pruning scheme, that is the combinatorial pruning rules (as in COMB) or the flow problem (as in FLOW) in such lazy manner significantly worsens the runtime of our algorithms and even leads to solving fewer instances in the time limit. 

Contrasting this, looking for an improving CD, i.e., executing Algorithm~\ref{alg:nonadjcliques}, only in every second, third or fifth node can improve the runtime of our algorithms. We refer to Table~\ref{tab:lazy} for an evaluation of the modified COMB algorithm for the random instances with $n=60$. Note that we do not regard the four instances which cannot be solved by COMB, thus the figures for COMB are different to the ones in Table~\ref{tab:all}. All remaining instances are solved by all the variants of COMB. As we can see in Table~\ref{tab:lazy}, a lazy application of the CD decreases the runtime of COMB.  This trend is more pronounced, the ``lazier'' the CD is updated. E.g., for $x=2$ (that is updating the CD at every second node in the  tree), the runtime decreases by $25\%$ whereas for $x=5$, the reduction exceeds 30 percent. The time saved in calculating the new CD is well worth the increase ($26\%$) in the number of nodes in the search tree when working with a potentially worse CD. We point out that the same holds for FLOW, but to a lesser extend. Note that we did not evaluate any further variants for $x>5$ as we speculate that, at some point, depending on the problem sizes, the improvements diminish and the algorithm becomes slower again.
\bigskip\\
As mentioned above, evaluating the pruning scheme in a lazy manner has a negative effect on the overall runtime of the algorithms. Combining this with the lazy update of the CD still has an overall negative impact. Thus, the ``best'', with respect to runtime, variation of the STD algorithm is to extend it with the combinatorial pruning rules and to apply a lazy update of the CD while explicitly evaluating the pruning rules at every node in the search tree.

\begin{table}
\caption{Variants of the COMB algorithm in which the \textbf{clique decomposition} is only updated every $x^{th}$ node in the  tree (random graphs, $n=60$).}  
\label{tab:lazy}
\setlength{\tabcolsep}{4pt}
{\footnotesize
\begin{center}
\begin{tabular}{l rrrrr rrrr}
\toprule
\multirow{2}{*}{\textbf{p}}&\multicolumn{4}{c}{\textbf{Time (s)}} & \multicolumn{4}{c}{\textbf{\# Nodes}} \\\cmidrule(lr){2-5}\cmidrule(lr){6-9}
 & \multirow{2}{*}{COMB} & \multicolumn{3}{c}{$x^{th}$ node} & \multirow{2}{*}{COMB} & \multicolumn{3}{c}{$x^{th}$ node} \\\cmidrule(lr){3-5}\cmidrule(lr){7-9}
 &  & \multicolumn{1}{c}{2} & \multicolumn{1}{c}{3} & \multicolumn{1}{c}{5} &  & \multicolumn{1}{c}{2} & \multicolumn{1}{c}{3} & \multicolumn{1}{c}{5}\\
\midrule
0.1&0.0&0.0&0.0&0.0&816.6&1,252.8&1,277.7&1,381.8\\
0.2&0.0&0.0&0.0&0.0&7,477.2&8,323.6&8,371.6&8,420.0\\
0.3&0.2&0.1&0.1&0.1&15,582.7&15,664.9&15,855.1&15,905.5\\
0.4&1.2&1.2&1.2&1.2&203,785.6&289,188.8&317,299.6&338,332.9\\
0.5&15.1&12.4&11.6&11.2&1,830,601.6&2,104,865.1&2,282,373.1&2,454,353.4\\
0.6&7.0&5.3&4.8&4.4&894,788.7&914,285.7&922,575.6&930,235.8\\
0.7&4.8&3.2&2.7&2.3&256,820.7&265,415.6&273,660.7&290,455.2\\
0.8&3.3&2.2&1.8&1.5&166,817.1&168,614.1&170,398.9&172,889.7\\
0.9&0.4&0.3&0.3&0.2&18,270.1&18,461.5&18,694.5&19,028.6\\
\textbf{Avg.}&\textbf{3.6}&\textbf{2.7}&\textbf{2.5}&\textbf{2.3}&\textbf{377,127.1}&\textbf{420,535.5}&\textbf{445,469.9}&\textbf{469,957.9}\\
\bottomrule
\end{tabular}
\end{center}
}
\end{table}

\subsection{DIMACS Instances} \label{results3}

In this subsection, we briefly discuss the behavior of our algorithms on the instances of the DIMACS challenge. All results of instances solvable within a time limit of $7,200$ seconds ($24$ in total) are displayed in Table~\ref{tab:dimacs}, all remaining instances cannot be solved by any of the three algorithms within the time limit. The DIMACS instances are particularly difficult for (standard) graph coloring problems and, apparently, they are also difficult for the ECP since only a fraction of the instances can be solved. We point out that in some cases, all algorithms require the same amount of nodes in the search tree, indicating the weakness of the (additional) pruning rules.
\bigskip\\
What we can observe in the table is that, on all instances, FLOW respectively COMB requires $7,291$ respectively $4,115$ seconds in total (excluding the $7,200$ seconds per unsolved instance), in comparison to the $10,123$ seconds required by STD. It is important to note that each of STD and FLOW cannot solve a different group of five instances. COMB cannot solve six instances. In total, COMB cannot solve all instances FLOW can solve but there are instances which COMB can solve but FLOW cannot. However, in contrast to the previous section, the FLOW algorithm seems to be preferable in this setting. 
\bigskip\\
However, the picture is mixed as all algorithms struggle with different and large parts of the instances. We believe that the weakness of FLOW, respectively COMB, in comparison to the findings in the previous section is due to the size and the low density of the regarded instances. 
In the case of large and sparse instances, the additional data structures required by FLOW (COMB) are rather time consuming to maintain. At the same time they yield little benefit as the pruning scheme is rather weak. 

All in all, it is clear that for the DIMCAS instances, more sophisticated implementations of the pruning rules become necessary as the feasibility of ``one size fits all'' approaches diminishes due to the problem size. Still, we believe that additional pruning rules are central for a successful algorithm regarding these instances.

\begin{table}
\caption{Statistics for the instances from the DIMCAS challenge ($p$ denotes the density of the instances).}
\label{tab:dimacs}
\begin{center}
\setlength{\tabcolsep}{2pt}
\renewcommand{\arraystretch}{0.75}
{\footnotesize
\begin{tabular}{l rr rrr rrr}
\toprule
\multirow{2}{*}{\textbf{Name}}&\multirow{2}{*}{$|V|$}&\multirow{2}{*}{p}&\multicolumn{3}{c}{\textbf{Time (s)}}&\multicolumn{3}{c}{\textbf{\# Nodes}}\\\cmidrule(lr){4-6}\cmidrule(lr){7-9}
&&&STD&FLOW&COMB&STD&FLOW&COMB\\
\midrule
1-FullIns\_3&30&0.23&0.0&0.0&0.0&842&480&679\\
1-Insertions\_4&67&0.10&3,243.5&-&-&2,550,140,072&-&-\\
2-FullIns\_3&52&0.15&69.6&462.7&118.1&68,439,880&44,908,973&57,865,454\\
2-Insertions\_3&37&0.11&0.0&0.5&0.1&20,240&15,353&15,420\\
3-Insertions\_3&56&0.07&6.9&183.4&28.3&7,016,020&5,579,518&5,619,199\\
DSJC125.1&125&0.09&2.3&3.1&1.0&1,386,303&111,546&258,735\\
fpsol2.i.1&469&0.09&0.1&1.4&0.7&5,741&5,741&5,741\\
inithx.i.1&864&0.05&-&20.2&-&-&40,604&-\\
le450\_15a&450&0.08&-&6,061.6&-&-&7,086,175&-\\
le450\_25a&480&0.07&-&0.7&0.2&-&3,860&3,860\\
miles1000&128&0.79&-&0.1&-&-&661&-\\
miles1500&128&0.64&0.0&0.0&0.0&10&5&7\\
mug100\_1&100&0.03&3,867.5&-&-&4,159,806,706&-&-\\
mug88\_1&88&0.04&840.2&-&2,096.6&1,014,959,534&-&842,473,020\\
mug88\_25&88&0.04&537.8&-&1,779.8&587,847,226&-&463,037,978\\
mulsol.i.1&197&0.20&0.0&0.2&0.1&1,264&1,264&1,264\\
myciel4&23&0.28&0.0&0.0&0.0&962&704&802\\
myciel5&47&0.22&0.5&13.0&1.6&385,726&267,726&269,610\\
queen6\_6&36&0.46&0.0&0.0&0.0&994&655&713\\
queen7\_7&49&0.40&0.0&0.5&0.1&10,352&8,131&8,656\\
queen8\_8&64&0.36&22.0&543.6&88.7&9,842,600&7,315,675&7,920,200\\
queen9\_9&81&0.65&1,532.9&-&-&583,479,058&-&-\\
school1\_nsh&352&0.24&-&0.7&0.4&-&918&918\\
zeroin.i.1&211&0.19&0.0&0.2&0.1&1,857&1,796&1,857\\
\bottomrule
\end{tabular}
}
\end{center}
\end{table}

\section{Conclusion}\label{sec:conclusion}

In this work, we have presented a flow based scheme for the generation of pruning rules for EQDSATUR. The scheme includes state of the art pruning rules as presented in~\cite{MNS:2014b} and extends them. The pruning decision is encoded in the flow problem, based on which further \textit{combinatorial} pruning rules (formulae) have been devised.

To evaluate our results, we added the new pruning rules to the EQDSATUR algorithm. We considered two variants, in the first, the flow problem is evaluated explicitly (FLOW), in the second we only consider a selection of combinatorial pruning rules (COMB). In our experiments, we have observed that, even a naive implementation of the pruning scheme via a flow problem (FLOW) is already sufficient for time competitiveness. However, an algorithm which relies on the combinatorial pruning rules (COMB) usually outperforms the naive approach, yielding the overall fastest algorithm. In all our random test instances COMB was the fastest, most stable algorithm. However, the DIMACS instances prove to be very difficult, respectively too large, for the ECP. There, the additional overhead of the pruning rules could not pay off, such that COMB, respectively FLOW could not significantly improve over the standard EQDSATUR algorithm (STD). In any case, the potential in the reduction of necessary nodes in the search tree is tremendous. 
\bigskip\\
All in all, the here presented pruning rules greatly expand the ones already known in the literature and can benefit the existing algorithmic approaches. That is, they can directly be incorporated into enumerative algorithms and can as well be used as cutting planes for MILP approaches, compare~\cite{MNS:2014a}. We advise to use them wherever applicable. However, with respect to the enumerative algorithms considered here, for large instances, more sophisticated implementations become necessary to avoid too costly operations for evaluating the pruning rules.
\bigskip\\

{\small \it
\textbf{Acknowledgement }
This work is partially supported by the German Federal Ministry of Education and Research (BMBF grant 05M13PAA, joint project 05M2013 - VINO: Virtual Network Optimization) as well as the Undergraduate Funds of the Excellence Initiative.

We thank our student assistants Sven F\"orster and Duc Thanh Tran for their work, especially regarding implementations and testing.
}

\end{document}